\documentclass[11pt,reqno]{amsart}
\usepackage{amssymb,mathrsfs,graphicx,subfigure, enumerate}
\usepackage{amsmath,amsfonts,amssymb,amscd,amsthm,bbm}
\usepackage{graphicx,colortbl,here}
\usepackage{extpfeil}

\usepackage[utf8]{inputenc}
\usepackage[T1]{fontenc}

\def\mi {{\mathrm i}}

\def\d {{\partial}}

\def \rpsi_i {|\psi_i \rangle}
\def \lpsi_i {\langle \psi_i|}
\def \lrpsi_i{\langle \psi_i | \psi_i \rangle}
\def \rpsi_k {|\psi_k \rangle}
\def \lpsi_k {\langle \psi_k|}
\def \lrpsi_k{\langle \psi_k | \psi_k \rangle}

\newcommand{\bbr}{\mathbb R}
\newcommand{\bbc}{\mathbb C}

\newcommand{\bbh}{\mathbb H}
\newcommand{\bbs}{\mathbb S}

\newcommand{\kp}{\kappa}
\newcommand{\dg}{\dagger}
\newcommand{\p}{\partial}

\renewcommand{\d}{{\textup{d}}}
\newcommand{\dt}{{\textup{d}t}}

\newcommand{\Skew}{\textup{Skew}}
\newcommand{\Sym}{\textup{Sym}}

\newcommand{\ba}{\begin{aligned}}
\newcommand{\ea}{\end{aligned}}

\newcommand{\be}{\begin{equation}}
\newcommand{\ee}{\end{equation}}

\newcommand{\tF}{\textup{F}}

\newtheorem{theorem}{Theorem}[section]
\newtheorem{lemma}{Lemma}[section]

\newtheorem{proposition}{Proposition}[section]
\newtheorem{remark}{Remark}[section]

\makeatletter
\@namedef{subjclassname@2020}{\textup{2020} Mathematics Subject Classification}
\makeatother
\begin{document}

\title[Asymptotic convergence of heterogeneous aggregation models]{Asymptotic convergence of heterogeneous first-order aggregation models: from the sphere to the unitary group}

\author[D. Kim]{Dohyun Kim}
\address[D. Kim]{\newline School of Mathematics, Statistics and Data Science,  \newline Sungshin Women's University, Seoul 02844, Republic of Korea}
\email{dohyunkim@sungshin.ac.kr}

\author[H. Park]{Hansol Park}
\address[Hansol Park]{\newline Department of Mathematics, \newline
Simon Fraser University, 8888 University Dr, Burnaby, BC V5A 1S6, Canada}
\email{hansol\_park@sfu.ca}
\email{hansol960612@snu.ac.kr}

\thanks{\textbf{Acknowledgment.}
The work of D. Kim was supported by the National Research Foundation of Korea (NRF) grant funded by the Korea government (MSIT) (No.2021R1F1A1055929) and the work of H. Park is supported by Pacific Institute for the Mathematical Science(PIMS), Canada postdoctoral fellowship.
}

\begin{abstract}
We provide the detailed asymptotic behavior for first-order aggregation models of heterogeneous oscillators. Due to the dissimilarity of natural frequencies, one could  expect that all relative distances converge to definite positive value and furthermore that each oscillator converges to a possibly different stationary point.   In order to establish the desired results, we introduce a novel method, called dimension reduction method that can be applied to a specific situation when the degree of freedom of the natural frequency is one. In this way, we would say that although a small perturbation is allowed, convergence toward an equilibrium of the gradient flow is still guaranteed. Several first-order aggregation models are provided as concrete examples by using the dimension reduction method to study the structure of the equilibrium, and numerical simulations are conducted to support theoretical results. 
\end{abstract}



\date{\today}

\maketitle

\tableofcontents


\section{Introduction} \label{sec:1}
\setcounter{equation}{0}

Collective behaviors of synchronous many-body systems have been widely studied after the work of two pioneers Kuramoto \cite{Ku1, Ku2} and Winfree \cite{Wi1, Wi2}. Depending on the dissimilarity of particles (or agents), a synchronous system is classified into two types: homogeneous and heterogeneous, and such dissimilarity is modeled as natural frequencies of particles. If natural frequencies are all the same, then particles are called identical and the corresponding system is said to be homogeneous. Otherwise, particles are called nonidentical and the system becomes heterogeneous. 

In literature, there have been lots of developments for  homogeneous systems compared to heterogeneous ones. 
As will be shown later, we use the Kuramoto model \cite{Ku1}, swarm sphere models on real and complex spheres \cite{H-P20, Lo10, O1} and the Lohe matrix model \cite{Lo09} as concrete examples. All these models with homogeneous structures can be represented as gradient flows with total relative distances as (analytic) potentials. Then, since underlying manifolds are compact, it follows from dynamical system theories that all solutions converge to stationary points. However, when a heterogeneous system is considered, natural frequencies are incorporating into the system and this leads to the breakdown of a gradient flow. This  causes numerous mathematical difficulties in analyzing asymptotic behaviors of the systems. That's why studying a  heterogeneous system is much harder than a homogeneous system. Precisely, when the gradient flow is slightly perturbed in the sense that
\begin{equation} \label{A-10}
\dot {\mathcal X} = -\kp \nabla_{\mathcal X} V(\mathcal X) + \varepsilon f(\mathcal X), \quad \mathcal X(t) \in \mathcal M, \quad t\geq0,
\end{equation}
where perturbation term $f(\mathcal X)$ is given to maintain the positive invariance of the underlying manifold (see \cite{H-J-K-P-Z} for another perturbation of a gradient flow). Note that equation \eqref{A-10} cannot be written as a gradient flow. Thus, convergence towards equilibrium would not be guaranteed. In addition, structure or existence of an equilibrium is not even known. In fact, if $\kp$ is sufficiently small, then equilibrium might not exist (see Proposition \ref{P2.3}). However, if we adopt our strategy, then we can show that a solution to \eqref{A-10} with some initial frameworks converges to an equilibrium, and a structure of the equilibrium is revealed.

The purpose of this paper is to analyzing several heterogeneous systems with our newly developed technique, so-called {\it dimension reduction method} which consists of the following two steps:

\noindent $\bullet$ (Step 1): reduce a given heterogeneous system into two subsystems: one is the heterogeneous Kuramoto system and the other is a homogeneous-like system.

\vspace{0.1cm}

\noindent $\bullet$ (Step 2): use the results for the heterogeneous Kuramoto system, such as the emergence of phase-locked states to investigate the asymptotic behavior of the original heterogeneous system.

In this work, we apply this technique to various first-order heterogeneous  aggregation models. As a simple example and motivation for later argument,  we here briefly introduce our method with a complex swarm sphere model in \cite{K-K21} as a simple example:
\begin{equation*} \label{A-11}
\dot z_j =  \mi H_jz_j + \frac\kp N \sum_{k=1}^N (z_k - \langle z_j,z_k\rangle z_j),\quad t>0,\quad j\in[N]:=\{1, 2, \cdots, N\},
\end{equation*}
where $z_j \in \bbh\bbs^{d-1}$ and $H_j$ is a $d\times d$ Hermitian matrix. First, we impose the conditions on natural frequencies so that their degree of freedom is one; that is, natural frequencies are only determined by one single parameter: for $a_i \in \bbr$ being a scalar, $H_j = a_j I_d$.  Of course, since degree of freedom for  a Hermitian matrix over $\bbr$ is $\frac{d(d+1)}{2}$, there arises loss in this sense. As compensate, we are able to associate the Kuramoto model whose natural frequencies are given as $\{a_i\}$:
\begin{equation} \label{A-12}
\dot\theta_j = a_j + \frac\kp N \sum_{k=1}^N \sin (\theta_k - \theta_j),\quad t>0,\quad j\in[N].
\end{equation} 
One of the most notable feature of \eqref{A-12} is that convergence toward an equilibrium can be achieved under a large coupling strength regime (see Proposition \ref{P2.1}). In other words, asymptotic behavior of a solution to \eqref{A-12} is somehow fully known. Then, we define an auxiliary intermediate variable $\xi_j$ which is defined as
\[
\xi_j(t) := z_j(t) e^{-\mi \theta_j(t)}.
\]
After deriving a temporal evolution of $\xi_j$, we will find a sufficient condition such that all $\xi_j$ converges to a common stationary point, say, $\xi^\infty$. By virtue of good convergence property for \eqref{A-12}, we combine two convergences to show that $z_j$ converges to $\xi^\infty e^{\mi\theta_j^\infty}$ whose modulus is turn out to be one.

Our main result of this paper is developing a novel method, namely, dimension reduction method in the above to study the asymptotic behavior of heterogeneous (or nonidentical) aggregation models. For detailed analysis,  we consider four different aggregation models: the real swarm sphere model, the Schr\"odinger--Lohe model, the Lohe Hermitian sphere model and the Lohe matrix model. The first result is concerned with  the swarm sphere model on the real unit sphere $\bbs^{d-1}$:
\begin{equation} \label{A-14}
\dot x_i =\Omega_i x_i + \frac\kp N \sum_{k=1}^N (x_k - \langle x_i,x_k\rangle x_i),
\end{equation} 
subject to the initial data $x_i(0) = x_i^0 \in \bbs^{d-1}$. Here, $\Omega_i$ is an $d\times d$ skew-symmetric matrix which plays a role of natural frequency of the $i$-th agent. Then, one can verify that $\bbs^{d-1}$ is positively invariant under \eqref{A-14}. We reformulate it to the model defined on the complex unit sphere $\bbh\bbs^{d-1}$ by means of the complexification (see Lemma \ref{L3.1}). It is worthwhile mentioning that it would not be possible to analyze \eqref{A-14}, since a skew-symmetric matrix is more harder to handle compared to a skew-Hermitian matrix. Thus, we find an equivalence reformulation of \eqref{A-14} in $\bbc$. Then, after associating the Kuramoto model and using aggregation estimates, one can deduce that a solution converges to equilibrium (see Theorem \ref{T3.1} and Theorem \ref{T3.2}). It should be noted that to the best of our knowledge, this is the first mathematically rigorous result for \eqref{A-14} with different $\{\Omega_i\}$ where convergence to equilibrium is established.

Second, we deal with two aggregation models on $\bbc$: the Schr\"odinger--Lohe model and the Lohe Hermitian sphere model. Similar to the previous circumstance, we use the dimension reduction method to show that solutions to both models converge to the equilibrium. Previously, there are few available results (e.g., \cite{H-H-K21}) for the convergence of relative distances, not the point itself.  However in this approach, existence of equilibrium for each configuration can be shown. 

Third, the models on real and complex unit spheres are lastly extended to the model on the unitary group, so-called the Lohe matrix model:
\[
\mi \dot U_i U_i^\dg = H_i + \frac{\mi\kp}{2N} \sum_{k=1}^N (U_k U_i^\dg - U_iU_k^\dg),\quad t>0,\quad i\in [N],
\] 
where $H_i$ is a $d\times d$ Hermitian matrix. We assume that there exist a common Hermitian matrix $H$ and scalars $\{a_i\}$ such that $H_i = H - a_i I_d$. 
We again apply the dimension reduction method to show that each element tends to a constant unitary matrix not necessarily the same under a strong coupling regime. Similar to the aforementioned argument, there is still a loss of degree of freedom; as a reward, we instead have fully analyzed the asymptotic behavior. 
It should be mentioned that the authors in \cite{H-R16} have intensively studied the Lohe matrix model by using orbital stability estimate. In this way, they can show that existence of limit for  relative composite matrices $\{U_i U_j^\dg\}$. However, that for each $\{U_i\}$ is still unknown. We extend this previous results by associating additional Kuramoto model with the help of the dimension reduction method.  \newline

The rest of this paper is organized as follows. In Section \ref{sec:2}, we review  previous results in relevant literature that will be used in later sections, and provide gradient flow formulation of the aggregation models. In Section \ref{sec:3}, the swarm sphere model on the real unit sphere is considered. In order to prepare the basic setting, we reformulate the swarm sphere model on $\bbr$ to the equivalent model but defined on $\bbc$. Then, we introduce our dimension reduction method to show that convergence toward equilibrium has been established. The Schr\"odinger--Lohe model and the Lohe Hermitian sphere model on $\bbc$ in Section \ref{sec:4} and the Lohe matrix model on the unitary group in Section \ref{sec:5} are discussed. By applying the same method, we obtain the same results in the previous section. Finally, Section \ref{sec:6} is devoted to a brief summary of the main results and some future direction.

\vspace{1cm}

\noindent {\bf Notation} (i) For real vectors $x,y \in \bbr^d$, we write 
\begin{equation*}
\langle x,y\rangle := \sum_{i=1}^d x^i y^i,\quad |x| := \sqrt{ \langle x,x\rangle}.
\end{equation*}
For complex valued vectors $w,z \in \bbc^d$, we write
\begin{equation*}
\langle w,z\rangle_\bbc := \sum_{i=1}^d w^i \bar z^i,\quad |w| := \sqrt {\langle w,w\rangle_\bbc}.
\end{equation*}
However, if there is no confusion and ambiguity, we simply write $\langle \cdot,\cdot\rangle_\bbc = \langle \cdot,\cdot\rangle$. \newline

\noindent (ii) For matrices, we denote $\mathcal M_d(\bbr)$ and $\mathcal M_d(\bbc)$ by the sets of all $d\times d$ matrices with real and complex entries, respectively. In addition, $\Sym_d(\bbr)$ is the set of all $d\times d$ symmetric matrices, $\Skew_d(\bbr)$  is the set of all $d\times d$ skew-symmetric matrices, and $\Skew_d(\bbc)$ is the set of all $d\times d$ skew-Hermitian matrices. Regarding the norm of matrices, we use the Frobenius norm: for $A \in \mathcal M_d(\bbc)$, 
\[
\| A\|_\tF = \sqrt{\textup{tr} (AA^\dg)}.
\]

\section{Preliminaries} \label{sec:2}
\setcounter{equation}{0}
In this section, we review the state-of-art results of first-order aggregation models on the real and complex spheres and the unitary group. For the purpose of this article, we restrict ourselves to heterogeneous oscillators. 

\subsection{Kuramoto model} 
First, we recall basic conservation law.  If we sum the relation \eqref{A-12} with respect to $j\in [N]$, then one can easily find 
\[
 \sum_{j=1}^N\dot \theta_j =\sum_{j=1}^N a_j.
\]
Hence, if the total sums of initial data and natural frequencies are zero, i.e., 
\begin{equation}  \label{B-0-0}
  \sum_{j=1}^N \theta_j(0) = \sum_{j=1}^N a_j = 0,
\end{equation} 
then the zero sum is conserved along the flow. 
\[
\sum_{j=1}^N \theta_j(t) = 0,\quad t>0. 
\]
In what follows, when the Kuramoto model is considered, we always assume \eqref{B-0-0}.  Although the Kuramoto model has been intensively and extensively investigated in literature, however for the concise presentation of the paper, we only introduce the results without proofs that will be used for our later argument.

\begin{proposition} \label{P2.1} 
Let $\{\theta_j\}$ be a solution to the following Cauchy problem:
\begin{equation*}
\begin{cases}
\displaystyle \dot{\theta}_j=a_j+\frac{\kappa}{N}\sum_{k=1}^N\sin(\theta_k-\theta_j),\quad t>0, \\
\theta_j(0) = \theta_j^0 \in \bbr,\quad j\in [N],
\end{cases} 
\end{equation*}
with initial data and natural frequencies satisfying \eqref{B-0-0}, and $\mathcal D(a)$ be the maximal diameter for natural frequencies:
\[
\mathcal D(a):= \max_{1\leq i,j\leq N} |a_i - a_j|.
\]
Then, the following assertions hold.

\noindent \textup{(i) (\cite{H-H-K10})}~~If $\mathcal D(\Theta^0)<\pi$, then the maximal diameter satisfies
\begin{equation*}
\frac\d\dt \mathcal D( \Theta) \leq \mathcal D(a) - \kp  \mathcal D( \Theta), \quad t>0.
\end{equation*} 
Thus, if initial diameter satisfy $\mathcal D(\Theta^0) < \frac{\mathcal D(a)}{\kp} <\pi$, then 
\[
 \mathcal D( \Theta(t)) < \frac{\mathcal D(a)}{\kp},\quad t>0.
\]
\textup{(ii) (\cite{H-R20})}~~If the coupling strength and initial data satisfy 
\[
 R_0 :=\sqrt{\frac{1}{N^2} \sum_{i,j=1}^N \cos (\theta_i^0 -\theta_j^0)} >0,\quad \kp > 1.6\frac{\mathcal D(a)}{R_0},
\]
then there exists $\{\theta_j^\infty\}$ such that
\begin{equation*}
\lim_{t\to\infty} \theta_j(t) = \theta_j^\infty,\quad j\in[N].
\end{equation*} 
\end{proposition}

\subsection{Real swarm sphere model}
In \cite{O1}, the author has proposed the consensus protocol on the unit sphere, and later, the author in \cite{Lo09} incorporated the natural frequencies into the swarming model on the unit sphere:
\begin{equation} \label{RSS}
\begin{cases}
\displaystyle \dot x_i = \Omega_i x_i + \frac\kp N \sum_{k=1}^N (x_k - \langle x_i,x_k\rangle x_i), \quad t>0, \\
\displaystyle x_i(0) =x_i^0 \in \bbs^{d-1},\quad i\in [N].
\end{cases}
\end{equation}
Here, $\Omega_i$ is a $d\times d$ skew-symmetric matrix, $\kp$ is a uniform coupling strength, and $\langle \cdot,\cdot\rangle$ is the standard $\ell^2$ inner produce in $\bbr^d$. 

\begin{proposition} \label{P2.2}
\noindent \textup{(i) (\cite{C-H14})}  Suppose that the coupling strength and initial data satisfy 
\begin{align*}
&\kp>4(d+1) \mathcal D(\Omega),\quad \min_{1\leq i \leq N} \langle x_i^0,x_c^0\rangle > 4(d+1)\mathcal D(\Omega),\\
& \rho \in (\alpha_-,\alpha_+),\quad \alpha_\pm:= \frac{ \kp\pm \sqrt{\kp^2 -2\kp(d+1)N \mathcal D(\Omega)} }{2\kp},
\end{align*}
and let $\{x_i\}$ be a solution to \eqref{RSS}. Then, we have
\[
\lim_{\kp\to\infty} \liminf_{t\to\infty} \rho(t) =1.
\]
\noindent \textup{(ii) (\cite{M-P-G})} We say that an equilibrium is referred to as dispersed if there is no open hemisphere containing the set $\{x_1,\cdots,x_N\}$. Then,  any dispersed equilibrium for \eqref{RSS} is (exponentially) unstable if the frequencies are small in the sense that
\[
\left(\sum_{i=1}^N \|\Omega_i\|_\tF^2\right)^\frac12 < \frac{\kp}{N+1} \left( N-1-\cos \frac\pi N \right) \left(1-\cos \frac\pi N \right).
\]
\end{proposition} 
We also refer the reader to \cite{Z-Z} for similar results in the above when the underlying graph becomes a digraph containing a directed spanning tree. 

\begin{remark}
In fact, as far as the authors know, there are no mathematically available results including those in Proposition \ref{P2.2} relevant to  the asymptotic formation of locked states for \eqref{RSS}. 
\end{remark}

\subsection{Complex swarm sphere model}
In \cite{Lo10}, the author proposed a coupled system of nonlinear Schr\"odinger equation, called the Schr\"odinger--Lohe model which reads as
\begin{equation} \label{SL}
\begin{cases}
\displaystyle \mi\p_t\psi_j = -\frac12\Delta \psi_j + V_j \psi_j +\frac{\mi\kp}{N} \sum_{k=1}^N (\psi_k - \langle \psi_j,\psi_k\rangle \psi_j), \quad t>0,\\
\displaystyle \psi_j(0,x) = \psi_j^0(x) \in L^2(\bbr^d),\quad j\in [N].
\end{cases}
\end{equation}
Here, $V_j$ is an one-body external potential acted on $j$-th node, $\kp$ is a uniform coupling strength and $\langle \cdot,\cdot\rangle$ is the standard inner product in $L^2(\bbr^d)$. 

\begin{proposition} \cite{H-H-K21} \label{P2.3}
Suppose that system parameters and initial data satisfy 
\[
V_j(x) = V(x) + \nu_j,\quad \kp >4\max_{1\leq i,j\leq N} |\nu_i-\nu_j|,\quad \max_{1\leq i,j \leq N} |1-\langle \psi_i^0,\psi_j^0\rangle| \ll1,
\]
and let $\{\psi_j\}$ be a solution to \eqref{SL}. Then, we have
\[
\lim_{t\to\infty} \langle \psi_i,\psi_j\rangle (t) \quad \textup{exists for $i,j\in [N]$.}
\]
\end{proposition}

Recently, the authors in \cite{H-P20} extended the swarm sphere model \eqref{RSS} defined on $\bbs^{d-1}$ to the model on the complex (or Hermitian) sphere denoted by $\bbh\bbs^{d-1}$:
\begin{equation} \label{LHS}
\begin{cases}
\displaystyle \dot z_j = \Omega_jz_j +\frac{\kp_0}{N} \sum_{k=1}^N \langle (z_k -\langle z_j,z_k\rangle z_j) + \frac{\kp_1}{N} \sum_{k=1}^N ( \langle z_k,z_j\rangle - \langle z_j,z_k\rangle)z_j, \quad t>0, \\
\displaystyle z_j(0) = z_j^0\in \bbh\bbs^{d-1},\quad j\in [N].
\end{cases}
\end{equation}
We also refer the reader to \cite{K-K21} for a new swarming model on the complex swarm sphere.


\begin{remark}
In fact, if one performs similar method used in Proposition \ref{P2.3} for \eqref{LHS} with a specific condition, say, $\Omega_j = \Omega + \mi \nu_j I_d$, then a similar result in Proposition \ref{P2.3} can be easily obtained without any effort. 
\end{remark}

\subsection{Lohe matrix model} \label{sec:2.4} 
In \cite{Lo09}, the authors introduced a first-order aggregation model on the unitary group of degree $d$ denoted as $\mathbf{U}(d)$:
\begin{equation} \label{LM}
\begin{cases}
\displaystyle \mi \dot U_j U_j^\dg = H_j + \frac{\mi \kp}{2N} \sum_{k=1}^N (U_k U_j^\dg - U_jU_k^\dg), \quad t>0, \\
U_j(0) = U_j^0 \in \mathbf{U}(d),\quad j\in [N].
\end{cases}
\end{equation}
Here, $H_j$ is a $d\times d$ Hermitian matrix that plays a role of the natural frequency, $\kp$ is a uniform coupling strength and $\dg$ denotes the Hermitian conjugate. 

\begin{proposition} \cite{H-R16} \label{P2.4}
Suppose that system parameters satisfy 
\[
\kp  > \frac{54}{17} \max_{1\leq i,j\leq N} \|H_i - H_j\|_\tF,\quad \max_{1\leq i,j\leq N} \|U_i^0 - U_j^0\|_\tF \ll1,
\]
and let $\{U_j\}$ be a solution to \eqref{LM}. Then, we have
\[
\lim_{t\to\infty} U_iU_j^\dg(t) \quad \textup{exists for $i,j\in [N]$.}
\]
\end{proposition}

\begin{remark}
In Proposition \ref{P2.4}, the authors guaranteed the limit of $U_i U_j^\dg$; however, the limit of each $U_j$ was unknown at that time. As will be shown later, we provide the limit of $U_j$. 
\end{remark}

\section{Convergence toward equilibrium on the real unit sphere} \label{sec:3} 
\setcounter{equation}{0}

In this section, we study the asymptotic behavior of the real swarm sphere model
\begin{equation} \label{C-1}
\begin{cases}
\vspace{0.1cm} \displaystyle \dot x_i = \Omega_i x_i + \frac\kp N \sum_{k=1}^N ( x_k - \langle x_i,x_k\rangle x_i),\quad t>0, \\
\displaystyle x_i(0) = x_i^0 \in \bbs^{d-1},\quad i\in [N].
\end{cases}
\end{equation} 

\subsection{Complexification}

In what follows, we find an equivalent formulation of \eqref{C-1} by means of `complexification.' We first consider the case of an even dimension.  For a given $\Omega_i \in \Skew_{2d} (\bbr)$, we decompose $\Omega_i$ into 
\begin{equation*}
\Omega_i = \begin{pmatrix} 
A_i & B_i \\ -B_i^\top & C_i 
\end{pmatrix}, \quad A_i,C_i\in \Skew_d(\bbr),\quad B_i\in\mathcal M_d(\bbr).
\end{equation*}
In addition for $x_i \in \bbr^{2d}$, we associate a $d$-dimensional complex vector $\omega_i \in \bbc^d$ defined as
\begin{equation*}
x_i =: (y_i,z_i) \in \bbr^d \times \bbr^d,\quad \omega_i := y_i + \mi z_i \in \bbc^d. 
\end{equation*}
Then, one can easily check 
\begin{equation*}
|\omega_i|^2 = |y_i|^2 + |z_i|^2 = |x_i|^2.
\end{equation*}

\begin{lemma} \label{L3.1} 
Let $(x_i,\Omega_i)$ be a solution to \eqref{C-1} that can be decomposed into   
\begin{equation} \label{B-4}
x_i = (y_i ,z_i) \in \bbr^d \times \bbr^d,\quad \Omega_i = \begin{pmatrix} 
A_i & B_i \\ -B_i^\top & C_i 
\end{pmatrix}  \in \textup{Skew}_{2d}(\bbr),\quad i\in [N]. 
\end{equation}
Suppose that the composition matrices $A_i,B_i$ and $C_i$ satisfy
\begin{equation} \label{B-5}
A_i = C_i \in \textup{Skew}_d(\bbr),\quad B_i \in \Sym_d(\bbr),\quad i\in[N]. 
\end{equation}
Then, the swarm sphere model \eqref{C-1} is equivalent to the following model on $\bbc^d$: for $\omega_i = y_i + \mi z_i $, 
\begin{equation} \label{B-6}
\dot \omega_i = \Xi_i \omega_i + \frac\kp N \sum_{k=1}^N ( \omega_k - \textup{Re}[\langle \omega_i,\omega_k\rangle] \omega_i ),\quad \Xi_i := A_i - \mi B_i.
\end{equation}
Moreover, the unit modulus is conserved along flow \eqref{B-6}. Hence, \eqref{B-6} is defined on $\bbh\bbs^{d-1}$.
\end{lemma}

\begin{proof}
We divide \eqref{A-14} as two subsystems: a free flow and an interaction flow. \newline

\noindent (i) For a free flow part, since $x_i$ and $\Omega_i$ satisfy \eqref{B-4}, we rewrite the free flow as
\begin{equation*}
\dot x_i = \Omega_i x_i \quad \Longleftrightarrow \quad \begin{pmatrix}
\dot y_i \\ \dot z_i 
\end{pmatrix} = \begin{pmatrix}
A_i y_i + B_i z_i \\ -B_i^\top y_i + C_i z_i
\end{pmatrix}.
\end{equation*}
Then, $\dot \omega_i$ satisfies
\begin{equation*}
\dot \omega_i = \dot y_i + \mi \dot z_i = (A_i y_i + B_i z_i ) + \mi (-B_i^\top y_i + C_i z_i ) = (A_i - \mi B_i^\top) y_i + (C_i - \mi B_i) (\mi z_i) .
\end{equation*}
Since $A_i,B_i$ and $C_i$ satisfy \eqref{B-5}, a free flow $\dot x_i = \Omega_i x_i$ in $\bbr^{2d}$ is equivalent to $\dot \omega_i = \Xi_i \omega_i$ in $\bbc^d$ with a skew-Hermitian matrix $\Xi_i = A_i - \mi B_i$. \newline 

\noindent (ii) On the other hand for an interaction flow, we observe
\begin{align*}
&\langle \omega_i,\omega_k\rangle_\bbc = \langle y_i+\mi z_i , y_k+ \mi z_k\rangle_\bbc = \langle y_i,y_k\rangle + \langle z_i,z_k\rangle + \mi (\langle z_i,y_k\rangle  -  \langle y_i,z_k\rangle).
\end{align*}
Since  the relation $\langle x_i,x_k\rangle = \langle y_i,y_k\rangle + \langle z_i,z_k\rangle$ holds, we find the relation between the complex and real inner products:
\begin{equation*}
\langle x_i,x_k\rangle = \textup{Re} [\langle \omega_i,\omega_k\rangle_\bbc].
\end{equation*}
Finally for the last assertion, the unit modulus property direct follows from the equivalence relation between \eqref{C-1} and \eqref{B-6}:
\begin{align*}
\frac{\d}{\dt} |\omega_i|^2 &= \langle \dot \omega_i,\omega_i\rangle + \langle  \omega_i,\dot \omega_i\rangle = \langle (A_i - \mi B_i )\omega_i,\omega_i \rangle + \langle \omega_i, (A_i - \mi B_i ) \omega_i \rangle \\
&\hspace{0.5cm}+ \frac\kp N \sum_{k=1}^N\Big( \langle \omega_k,\omega_i \rangle - \textup{Re}[\langle \omega_i,\omega_k\rangle ] |\omega_i|^2 + \langle \omega_i,\omega_k\rangle -  \textup{Re}[\langle \omega_i,\omega_k\rangle ] |\omega_i|^2\Big)
\\
& = \langle A_i\omega_i ,\omega_i \rangle + \langle \omega_i , A_i\omega_i\rangle + \langle -\mi B_i \omega_i,\omega_i\rangle + \langle \omega_i , -\mi B_i \omega_i\rangle \\
&\hspace{0.5cm} + \frac{2\kp}{N} \sum_{k=1}^N \textup{Re}[\langle \omega_i,\omega_k\rangle ] ( 1- |\omega_i|^2) \\
& =\langle (A_i + A_i^\top)\omega_i,\omega_i\rangle + \langle \mi(B_i^\top - B_i) \omega_i,\omega_i\rangle  + \frac{2\kp}{N} \sum_{k=1}^N \textup{Re}[\langle \omega_i,\omega_k\rangle ] ( 1- |\omega_i|^2) \\
&=\frac{2\kp}{N} \sum_{k=1}^N \textup{Re}[\langle \omega_i,\omega_k\rangle ] ( 1- |\omega_i|^2).
\end{align*}
Hence, we have
\[
1-|\omega_i(t)|^2 = ( 1-|\omega_i^0|^2 ) \exp \left(-\frac{2\kp}{N} \sum_{k=1}^N \textup{Re}[\langle \omega_i,\omega_k\rangle ] \right),
\]
which gives the desired result. 
\end{proof}

On the other hand for the odd dimension, we will basically extend  $\bbr^{2d+1}$ to $\bbr^{2d+2}$ and apply the same argument of an even dimensional case. Let $x_i \in \bbr^{2d+1}$  and $\Omega_i \in \textup{Skew}_{2d+1}(\bbr)$. Then, we consider an augmented matrix of $\Omega_i$, denoted by $\tilde \Omega_i$: 
\begin{equation*}
\tilde \Omega_i = \begin{pmatrix} 
\Omega_i & \mathbf{0}_{2d+1} \\
\mathbf{0}_{2d+1}^\top & 0 
\end{pmatrix} \in \textup{Skew}_{2d+2}(\bbr).
\end{equation*}
Similarly, we set an augmented vector of $x_i$ denoted as $\tilde x_i$:
\begin{equation*}
\tilde x_i := (x_i,0) \in \bbr^{2d+2}.
\end{equation*} 
Since $(x_i,\Omega_i)$ satisfies \eqref{A-14}, we see that $(\tilde x_i, \tilde \Omega_i)$ satisfies
\begin{equation*}
\dot{\tilde x}_i = \tilde \Omega_i \tilde x_i + \frac\kp N \sum_{k=1}^N ( \tilde x_k - \langle \tilde x_i, \tilde x_k\rangle \tilde x_i),
\end{equation*}
which coincides \eqref{A-14} in $\bbr^{2d+2}$. Then, we denote
\begin{equation*}
\tilde x_i = (y_i,z_i) \in \bbr^{d+1} \times \bbr^{d+1}, \quad \tilde \Omega_i = \begin{pmatrix} 
A_i & B_i \\ -B_i^\top & C_i 
\end{pmatrix},
\end{equation*}
where $A_i,C_i \in \textup{Skew}_{d+1}(\bbr)$ and $B_i \in \mathcal M_{d+1}(\bbr)$. Then, $\omega_i = y_i + \mi z_i$ satisfies \eqref{B-6}. \newline

In the following subsections, we use the complexified swarm sphere model \eqref{B-6} instead of the real swarm sphere model \eqref{C-1}.

\subsection{Dimension reduction method}
In this subsection, we introduce a novel method to investigate the asymptotic behavior for heterogeneous aggregation models. First, we assume that the natural frequency matrices satisfy the following zeroth-order approximation condition: precisely, there exists $a_j \in \bbr$ such that
\[
\Xi_j = \Xi + \mi a_j I_d,\quad j\in [N].
\]
Due to the hermitian symmetry, one would set
\[
\Xi = O_d. 
\]
Then, our system reads as
\begin{equation}  \label{C-10}
\dot{w}_j=\mathrm{i}a_jw_j+\frac{\kappa}{N}\sum_{k=1}^N\left(w_k-\mathrm{Re}\left(\langle w_k, w_j\rangle\right) w_j\right),\quad t>0,\quad j\in [N].
\end{equation}
For $a_i$ that appears in \eqref{C-10}, we associate the Kuramoto model whose natural frequencies are exactly same as $\{a_i\}$. To be more specific, let $\{\theta_i\}$ be a solution to 
\begin{equation} \label{C-10-1}
\dot{\theta}_j=a_j+\frac{\kappa}{N}\sum_{k=1}^N\sin(\theta_k-\theta_j),\quad t>0,\quad j\in [N],
\end{equation}
together with initial data and natural frequencies satisfying zero sum conditions:
\begin{equation*} 
\sum_{j=1}^N \theta_j(0)=0 \quad \textup{and} \quad \sum_{j=1}^N a_j =0.
\end{equation*}
Thus,  we would say that two models \eqref{C-10} and \eqref{C-10-1} are weakly coupled through natural frequencies. We define a vector 
\[
z_j(t) : = w_j (t) e^{-\mi \theta_j(t)} \in \bbh\bbs^{d-1} \subseteq\bbc^d,\quad t>0.
\]
Below, we find the governing equation for $z_i$. 

\begin{lemma}
Let $w_j$ and $\theta_j$ be solutions to \eqref{C-10} and \eqref{C-10-1}, respectively. Then, $z_i = w_i e^{-\mi \theta_i}$ satisfies
\begin{equation} \label{C-11}
\dot z_j = \frac{\kappa}{N}\sum_{k=1}^N\left(
e^{\mathrm{i}(\theta_k-\theta_j)}(z_k-z_j)-\mathrm{Re}(\langle e^{\mathrm{i}(\theta_k-\theta_j)}(z_k-z_j), z_j\rangle)z_j\right),\quad t>0,\quad j\in [N]. 
\end{equation}
\end{lemma} 
\begin{proof}
By straightforward calculation, one finds the desired dynamics for $z_i$: 
\begin{align*}
\dot{z}_j&=\mathrm{i}a_j z_j+\frac{\kappa}{N}\sum_{k=1}^N\left(
e^{\mathrm{i}(\theta_k-\theta_j)}z_k-\mathrm{Re}(\langle e^{\mathrm{i}(\theta_k-\theta_j)}z_k, z_j\rangle )z_j
\right)-\mathrm{i}\left(a_j+\frac{\kappa}{N}\sum_{k=1}^N\sin(\theta_k-\theta_j)\right)z_j\\
&=\frac{\kappa}{N}\sum_{k=1}^N\left(
e^{\mathrm{i}(\theta_k-\theta_j)}z_k-\left(\mathrm{Re}(\langle e^{\mathrm{i}(\theta_k-\theta_j)}z_k, z_j\rangle )+\mathrm{i}\sin(\theta_k-\theta_j)\right)z_j
\right)\\
&=\frac{\kappa}{N}\sum_{k=1}^N\left(
e^{\mathrm{i}(\theta_k-\theta_i)}(z_k-z_i)-\mathrm{Re}(\langle e^{\mathrm{i}(\theta_k-\theta_i)}(z_k-z_i), z_i\rangle)z_i\right).
\end{align*}
\end{proof} 

\begin{remark}
If we write $\tilde z_{jk} := e^{\mi (\theta_k - \theta_j)} (z_k -z_j)$, then \eqref{C-11} is written as
\[
\dot z_j = \frac\kp N \sum_{k=1}^N ( \tilde z_{jk} - \textup{Re} \langle z_j,\tilde z_{jk}\rangle z_j  ).
\]
It is worthwhile mentioning that dynamics of $z_j$ does not contain free flow; instead, natural frequency part $a_j$ is encoded into $\tilde z_{jk}$.
\end{remark}

For later use, we denote
\[
h_{ij}:= \langle z_i,z_j\rangle,\quad R_{ij} := \textup{Re}h_{ij},\quad \mathcal D( \mathcal R) := \max_{1\leq i,j\leq N} |1-R_{ij}|. 
\]
Since our goal is to find a sufficient condition leading to the zero convergence of $\mathcal D(\mathcal R)$, we derive the dynamics of $R_{ij}$.

\begin{lemma}
Let $\{z_j\}$ be a solution to \eqref{C-11}. Then, $R_{ij}$ satisfies 
\begin{align} \label{C-12}
\begin{aligned}
\frac\d\dt  (1-R_{ij} ) &= -2\kp(1-R_{ij}) + \frac\kp N \sum_{k=1}^N (1-R_{ik} + 1-R_{kj})(1-R_{ij}) \\ 
& \hspace{0.5cm}-\frac\kp N \sum_{k=1}^N\Big[ (\textup{Re} \langle \mathcal A, z_j\rangle + \textup{Re} \langle \mathcal B, z_i\rangle)(1-R_{ij} ) + \textup{Re} \langle \mathcal A-\mathcal B, z_j-z_i\rangle\Big],
\end{aligned}
\end{align}
where $\mathcal A$ and $\mathcal B$ are defined as
\begin{equation} \label{C-12-1}
\mathcal A := (e^{\mathrm{i}(\theta_k-\theta_i)}-1)(z_k-z_i), \quad \mathcal B := (e^{\mathrm{i}(\theta_k-\theta_j)}-1)(z_k-z_j). 
\end{equation}  
In addition, the maximal diameter $\mathcal D(\mathcal R)$ satisfies
\begin{equation} \label{C-13}
\frac\d\dt \mathcal D(\mathcal R) \leq-\kp \left( 1 - 4\sin \frac{\mathcal D(\Theta)}{2}   \right) \mathcal D(\mathcal R) +2\kp\left(\left| \sin \frac{\mathcal D(\Theta)}{2}\right| +\frac12\right) \mathcal D(\mathcal R)^2  , \quad t>0. 
\end{equation} 
\end{lemma}

\begin{proof}
We observe 
\begin{align*}
\frac{\d}{\dt}\mathrm{Re}\langle z_i, z_j\rangle&=\frac{\kappa}{N}\sum_{k=1}^N\left(
\mathrm{Re}\langle e^{\mathrm{i}(\theta_k-\theta_i)}(z_k-z_i), z_j\rangle-\mathrm{Re}(\langle e^{\mathrm{i}(\theta_k-\theta_i)}(z_k-z_i), z_i\rangle)\mathrm{Re}\langle z_i, z_j\rangle
\right)\\
&\hspace{0.5cm}+\frac{\kappa}{N}\sum_{k=1}^N\left(
\mathrm{Re}\langle z_i, e^{\mathrm{i}(\theta_k-\theta_j)}(z_k-z_j)\rangle-\mathrm{Re}(\langle e^{\mathrm{i}(\theta_k-\theta_j)}(z_k-z_j), z_j\rangle)\mathrm{Re}\langle z_i, z_j\rangle
\right) \\
&= : \frac\kp N \sum_{k=1}^N ( \mathcal I_{11} - \mathcal I_{12} ) + \frac\kp N \sum_{k=1}^N ( \mathcal I_{13} - \mathcal I_{14} ).
\end{align*}
Below, we consider $\mathcal I_{1k}$ ($k=1,2,3,4$), respectively. \newline

\noindent $\bullet$ (Calculation of $\mathcal I_{12} + \mathcal I_{14}$): we observe
\begin{align*}
\mathcal I_{11} + \mathcal I_{13} &=\mathrm{Re}\langle e^{\mathrm{i}(\theta_k-\theta_i)}(z_k-z_i), z_j\rangle + \mathrm{Re}\langle z_i, e^{\mathrm{i}(\theta_k-\theta_j)}(z_k-z_j)\rangle \\
& =R_{kj} - R_{ij} + R_{ik} - R_{ij} \\
&\hspace{0.5cm}+ \mathrm{Re}\langle (e^{\mathrm{i}(\theta_k-\theta_i)}-1)(z_k-z_i), z_j\rangle + \mathrm{Re}\langle z_i, (e^{\mathrm{i}(\theta_k-\theta_j)}-1)(z_k-z_j)\rangle.
\end{align*}

\noindent $\bullet$ (Calculation of $\mathcal I_{11} + \mathcal I_{13}$): similarly, one has

\begin{align*}
&\mathcal I_{12} + \mathcal I_{14} \\
& = \mathrm{Re}(\langle e^{\mathrm{i}(\theta_k-\theta_i)}(z_k-z_i), z_i\rangle)\mathrm{Re}\langle z_i, z_j\rangle  + \mathrm{Re}(\langle e^{\mathrm{i}(\theta_k-\theta_j)}(z_k-z_j), z_j\rangle)\mathrm{Re}\langle z_i, z_j\rangle \\
& = (R_{ik}-1)R_{ij} + (R_{kj}-1)R_{ij}  \\
& \hspace{0.5cm}+ \mathrm{Re}(\langle (e^{\mathrm{i}(\theta_k-\theta_i)}-1)(z_k-z_i), z_i\rangle)\mathrm{Re}\langle z_i, z_j\rangle  + \mathrm{Re}(\langle (e^{\mathrm{i}(\theta_k-\theta_j)}-1)(z_k-z_j), z_j\rangle)\mathrm{Re}\langle z_i, z_j\rangle \\
&= (1-R_{ik} + 1-R_{kj})(1-R_{ij}) - (1-R_{ik} + 1-R_{kj}) \\
&\hspace{0.5cm}+ \mathrm{Re}(\langle (e^{\mathrm{i}(\theta_k-\theta_i)}-1)(z_k-z_i), z_i\rangle)\mathrm{Re}\langle z_i, z_j\rangle  - \mathrm{Re}(\langle (e^{\mathrm{i}(\theta_k-\theta_j)}-1)(z_k-z_j), z_j\rangle)\mathrm{Re}\langle z_i, z_j\rangle.
\end{align*}
We combine the calculations above to find 
\begin{align*}
&\mathcal I_{11} - \mathcal I_{12} + \mathcal I_{13} - \mathcal I_{14} \\
&= R_{kj} - R_{ij} + R_{ik} - R_{ij} + (1-R_{ik} + 1-R_{kj}) -(1-R_{ik} + 1-R_{kj})(1-R_{ij})  \\
&\hspace{0.5cm} + \mathrm{Re}\langle (e^{\mathrm{i}(\theta_k-\theta_i)}-1)(z_k-z_i), z_j\rangle + \mathrm{Re}\langle z_i, (e^{\mathrm{i}(\theta_k-\theta_j)}-1)(z_k-z_j)\rangle \\
&\hspace{0.5cm} -  \mathrm{Re}(\langle (e^{\mathrm{i}(\theta_k-\theta_i)}-1)(z_k-z_i), z_i\rangle)\mathrm{Re}\langle z_i, z_j\rangle  + \mathrm{Re}(\langle (e^{\mathrm{i}(\theta_k-\theta_j)}-1)(z_k-z_j), z_j\rangle)\mathrm{Re}\langle z_i, z_j\rangle \\
& = 2(1-R_{ij}) -(1-R_{ik} + 1-R_{kj})(1-R_{ij})  \\
&\hspace{0.5cm} + \mathrm{Re}\langle (e^{\mathrm{i}(\theta_k-\theta_i)}-1)(z_k-z_i), z_j\rangle + \mathrm{Re}\langle z_i, (e^{\mathrm{i}(\theta_k-\theta_j)}-1)(z_k-z_j)\rangle \\
&\hspace{0.5cm} -  \mathrm{Re}(\langle (e^{\mathrm{i}(\theta_k-\theta_i)}-1)(z_k-z_i), z_i\rangle)\mathrm{Re}\langle z_i, z_j\rangle  - \mathrm{Re}(\langle (e^{\mathrm{i}(\theta_k-\theta_j)}-1)(z_k-z_j), z_j\rangle)\mathrm{Re}\langle z_i, z_j\rangle .
\end{align*} 
If we recall the definition of $\mathcal A$ and $\mathcal B$ in \eqref{C-12-1}, then one has 
\begin{align*}
&\mathcal I_{11} - \mathcal I_{12} + \mathcal I_{13} - \mathcal I_{14} \\
 &\hspace{0.5cm} = R_{kj} - R_{ij} + R_{ik} - R_{ij} + (1-R_{ik} + 1-R_{kj}) -(1-R_{ik} + 1-R_{kj})(1-R_{ij}) \\
&\hspace{1cm} +\textup{Re} \langle \mathcal A, z_j\rangle - \textup{Re} \langle \mathcal A,z_i \rangle R_{ij} + \textup{Re} \langle \mathcal B, z_i\rangle - \textup{Re} \langle \mathcal B,z_j \rangle R_{ij} \\
& =R_{kj} - R_{ij} + R_{ik} - R_{ij} + (1-R_{ik} + 1-R_{kj}) -(1-R_{ik} + 1-R_{kj})(1-R_{ij}) \\
&\hspace{1cm} +(\textup{Re} \langle \mathcal A, z_j\rangle + \textup{Re} \langle \mathcal B, z_i\rangle)(1-R_{ij} ) + \textup{Re} \langle \mathcal A-\mathcal B, z_j-z_i\rangle.
\end{align*}
Consequently, we can derive the desired dynamics \eqref{C-12} for $R_{ij}$.  

For the second assertion, we write 
\begin{align*}
\mathcal I_{15} := (\textup{Re} \langle \mathcal A, z_j\rangle + \textup{Re} \langle \mathcal B, z_i\rangle)(1-R_{ij} ), \quad \mathcal I_{16} :=  \textup{Re} \langle \mathcal A-\mathcal B, z_j-z_i\rangle.
\end{align*}
Below, we present estimates for $\mathcal I_{15}$ and $\mathcal I_{16}$, respectively. \newline

\noindent $\bullet$ (Estimate of $\mathcal I_{15}$): we observe
\begin{equation} \label{C-18}
|e^{\mi \varphi} - 1| \leq 2\left|\sin \frac\varphi2\right|,
\end{equation} 
which yields 
\begin{equation*}
|\mathcal A | = |(e^{\mathrm{i}(\theta_k-\theta_i)}-1)(z_k-z_i)|  \leq 2\left|\sin \frac{\mathcal D(\Theta)}{2}\right| \mathcal D(\mathcal R),\quad |\mathcal B |\leq  2\left|\sin \frac{\mathcal D(\Theta)}{2}\right| \mathcal D(\mathcal R).
\end{equation*} 
Hence, one finds
\begin{align*}
\mathcal I_{15} =  (\textup{Re} \langle \mathcal A, z_j\rangle + \textup{Re} \langle \mathcal B, z_i\rangle)(1-R_{ij} ) \leq 4\left|\sin \frac{\mathcal D(\Theta)}{2}\right| \mathcal D(\mathcal R) \cdot \frac{\mathcal D(\mathcal R)^2}{2} = 2\left|\sin \frac{\mathcal D(\Theta)}{2}\right| \mathcal D(\mathcal R)^3.
\end{align*}

\noindent $\bullet$ (Estimate of $\mathcal I_{16}$): We see
\[
\mathcal A-\mathcal B = (e^{\mi(\theta_k-\theta_i)}-1)(z_j-z_i) +(e^{\mi(\theta_k-\theta_i)} - e^{\mi(\theta_k-\theta_j)})(z_k-z_j).
\]
Thus, we obtain
\begin{align*}
\mathcal I_{16} &= \textup{Re} \langle \mathcal A-\mathcal B, z_j-z_i\rangle =(\cos(\theta_k-\theta_i)-1) |z_j-z_i|^2  \\
&\hspace{0.5cm}+ \textup{Re} \langle (e^{\mi(\theta_k-\theta_i)} - e^{\mi(\theta_k-\theta_j)})(z_k-z_j), z_j-z_i\rangle \\
&\leq 4\sin \frac{\mathcal D(\Theta)}{2}   \mathcal D(\mathcal R)^2 .
\end{align*}
Since $1-R_{ij}$ satisfies 
\begin{align} \label{C-20}
\begin{aligned}
\frac\d\dt  (1-R_{ij} ) &= -2\kp(1-R_{ij}) + \frac\kp N \sum_{k=1}^N (1-R_{ik} + 1-R_{kj})(1-R_{ij}) \\ 
& -\frac\kp N \sum_{k=1}^N\Big[ (\textup{Re} \langle \mathcal A, z_j\rangle + \textup{Re} \langle \mathcal B, z_i\rangle)(1-R_{ij} ) + \textup{Re} \langle \mathcal A-\mathcal B, z_j-z_i\rangle\Big]  \\
& = -2\kp(1-R_{ij}) + \frac\kp N \sum_{k=1}^N (1-R_{ik} + 1-R_{kj})(1-R_{ij}) \\
& \hspace{0.5cm} -\frac\kp N \sum_{k=1}^N \Big[ \mathcal I_{15} + \mathcal I_{16} \Big],
\end{aligned}
\end{align}
in \eqref{C-20}, we collect all estimates to derive
\begin{align*} \label{C-25}
\begin{aligned}
\frac\d\dt \mathcal D(\mathcal R) &\leq -\kp \left( 1 - 4\sin \frac{\mathcal D(\Theta)}{2}   \right) \mathcal D(\mathcal R) +2\kp\left| \sin \frac{\mathcal D(\Theta)}{2}\right|\mathcal D(\mathcal R)^2   +\frac{\kp}{2}\mathcal D(\mathcal R)^3 \\
&\leq-\kp \left( 1 - 4\sin \frac{\mathcal D(\Theta)}{2}   \right) \mathcal D(\mathcal R) +2\kp\left(\left| \sin \frac{\mathcal D(\Theta)}{2}\right| +\frac12\right) \mathcal D(\mathcal R)^2  .
\end{aligned}
\end{align*} 
\end{proof}

We are now ready to provide a sufficient condition under which $\mathcal D(\mathcal R)$ tends to zero. Furthermore, all $z_j$ tends to a common stationary point. 

\begin{proposition} \label{P3.1} 
Suppose that the coupling strength (sufficiently large in the sense below) and natural frequencies satisfy  
\begin{equation} \label{C-26}
\kp > 2\mathcal D(a),\quad \sum_{j=1}^N a_j=0,
\end{equation}
and that there exists a positive constant $\delta \in (0,\frac12)$ such that
\begin{align} \label{C-27}
\begin{aligned}
\mathcal D(\Theta^0)<\delta<\frac{\mathcal D(a)}{\kp}<\frac12\quad \textup{and} \quad \mathcal D(\mathcal R^0)< \frac{1-2\delta}{1+\delta}.
\end{aligned}
\end{align}
Let $\{z_i\}$ be a solution to \eqref{C-11} where $\{\theta_i\}$ is a solution to \eqref{C-10-1}. Then, the maximal diameter $\mathcal D(\mathcal R)$ converges to zero with an exponential convergence rate:
\[
\lim_{t\to\infty} \mathcal D(\mathcal R(t))=0.
\]
In addition, there exists $z^\infty \in \bbh\bbs^{d-1}$ such that
\begin{equation} \label{C-28}
\lim_{t\to\infty} z_i(t)  = z^\infty,\quad i\in [N]. 
\end{equation}
\end{proposition}

\begin{proof}
Since we assume \eqref{C-26} and $\eqref{C-27}_1$, it follows from Proposition \ref{P2.1} that inequality \eqref{C-13} for $\mathcal D(\mathcal R)$ becomes
\[
\frac\d\dt \mathcal D(\mathcal R) \leq -\kp (1-2\delta)  \mathcal D(\mathcal R) + \kp (1+ \delta)  \mathcal D(\mathcal R)^2.
\]
Then, initial condition $\eqref{C-27}_2$ gives the desired exponential convergence. For the second assertion, we use the governing dynamics \eqref{C-11} for $z_i$ to obtain
\[
|\dot z_j | \leq 2\kp \mathcal D( \mathcal R), \quad t>0,
\]
which gives that $\dot z_j \in L^1(\bbr_+)$. Thus, one finds
\[
\lim_{t\to\infty} z_j(t) = z_j^0 + \int_0^\infty \dot z_j(s) \d s =:z_j^\infty.
\]
Moreover, since complete aggregation occurs, we observe for $i\neq j$, 
\[
|z_i^\infty - z_j^\infty| \leq \lim_{t\to\infty} \Big( |z_i^\infty - z_i(t) | + |z_i(t) - z_j(t)| + |z_j(t) - z_j^\infty| \Big) =0.
\]
Hence, there exists $z^\infty$ satisfying \eqref{C-28}. 
\end{proof}

In Proposition \ref{P3.1}, we have investigated the asymptotic behavior of $z_i$. In the following theorem, we use Theorem \ref{P3.1} to study the asymptotic behavior of $w_i$ that is governed by \eqref{C-10}. For this, we write
\[
\mathcal D(W(t)):= \max_{1\leq i,j\leq N } |w_i(t) - w_j(t)|,\quad t>0. 
\]

\begin{theorem} \label{T3.1} 
Suppose that the coupling strength and initial data satisfy 
\[
\kp > 2\mathcal D(a), \quad \sum_{j=1}^N a_j = 0
\]
and that there exists $\delta \in (0,\beta)$ such that 
\begin{equation} \label{C-30}
0<\delta <\frac{\mathcal D(a)}{\kp}<\beta,\quad \mathcal D(W^0) < \sqrt{\frac{2(1-2\delta)}{1+\delta}} -2\sin \frac\delta 2,
\end{equation} 
where $\beta \approx 0.434$ is a unique positive root for $\sqrt{\frac{2(1-2s)}{1+s}} = 2\sin \frac s2$. 
Let $\{w_j\}$ be a solution to \eqref{C-10}. Then, there exists $w_j^\infty\in \mathbb{H}\bbs^{d-1}$ such that
\[
\lim_{t\to\infty} w_j(t) = w_j^\infty,\quad j\in [N].
\]
In addition,   $\{w_j^\infty\}$ satisfy 
\[
|\langle w_i^\infty,w_j^\infty\rangle|=1,\quad i,j\in [N].
\]
\end{theorem}

\begin{proof}
In order to adapt Proposition \ref{P3.1}, we use the initial frameworks in \eqref{C-26} and \eqref{C-27} to see
\begin{align*}
|z_i^0 - z_j^0 | = |w_i^0 - e^{\mi(\theta_i^0 - \theta_j^0)}w_j^0| \leq |w_i^0 - w_j^0| + |e^{\mi(\theta_i^0 - \theta_j^0)}|-1|,
\end{align*}
which yields 
\begin{align*}
\sqrt{2\mathcal D(\mathcal R^0)} \leq \mathcal D(W^0) + 2\sin \frac{\mathcal D(\Theta^0)}{2} < \mathcal D(W^0)  + 2\sin \frac\delta 2.
\end{align*}
Thus, if  $\mathcal D(W^0)$ satisfy \eqref{C-30}, then initial conditions for  $\mathcal D(\mathcal R^0)$ and $\mathcal D(\Theta^0)$ are fulfilled. Consequently, it follows from Theorem \ref{T3.1} and Proposition \ref{P2.1} that there exists $z^\infty$ and $\theta^\infty$ such that
\[
\lim_{t\to\infty} z_j(t) = z^\infty,\quad \lim_{t\to\infty} \theta_j(t) = \theta_j^\infty. 
\]
Hence, we have
\[
\lim_{t\to\infty} w_j(t) = z^\infty e^{\mi \theta_j^\infty} =: w_j^\infty.
\]
Moreover, 
\[
\langle w_i^\infty,w_j^\infty\rangle = \langle  z^\infty e^{\mi \theta_i^\infty},  z^\infty e^{\mi \theta_j^\infty}\rangle = e^{\mi(\theta_i^\infty- \theta_j^\infty)}
\]
which has a unit modulus. 
\end{proof}

\begin{remark} \label{R3.2} 
We mention that $w_i(t)$ is asymptotically decomposed into two parts: $z^\infty$ as a common axis part and $e^{\mi\theta_i^\infty}$ as its own rotation part.  Heuristically, we would say that the asymptotic behavior of a high dimensional object $w_i(t)$ is determined by a one-dimensional object $\theta_i^\infty$. 
\end{remark}

In fact, we have studied the asymptotic behavior of $w_i$ on $\mathbb{H} \bbs^{d-1}$ by using auxiliary variables $z_i$ and $\theta_i$. However, we are essentially concerned with the asymptotic behavior of \eqref{C-1} where a solution $\{x_i\}$ to \eqref{C-1} lies on $\bbs^{d-1}$. Thus, our last step in this section is dedicated to representing Theorem \ref{T3.1} in terms of $x_i$.  

\begin{theorem} \label{T3.2} 
Consider the real swarm sphere model \eqref{C-1} on even dimension $2d$ whose solution is written as $x_j=(x_j^1,\cdots,x_j^{2d}) \in \bbs^{2d-1}$. Suppose that initial data and system parameters satisfy the following conditions: 
\begin{align*}
&\textup{(i)}~~ \kp >2\mathcal D(a), \quad 0<\delta<\frac{\mathcal D(a)}{\kp}<\beta. \\
&\textup{(ii)}~~ \Omega_j = \begin{pmatrix} A & B \\ - B & A \end{pmatrix} + a_j \begin{pmatrix} O_d &  I_d \\ - I_d & O_d  \end{pmatrix},\quad \forall~j\in[N],\\
& \textup{(iii)}~~ \max_{1\leq i,j\leq N}\Big( |(x_i^{1,0},\cdots, x_i^{d,0})-(x_j^{1,0},\cdots, x_j^{d,0})| + |(x_i^{d+1,0},\cdots, x_i^{2d,0})-(x_j^{d+1,0},\cdots, x_j^{2d,0})|\Big) \\
& \hspace{1cm}<\sqrt{\frac{2(1-2\delta)}{1+\delta}} -2\sin \frac\delta 2,
\end{align*}
where $A$ and $B$ are $d\times d$ skew-symmetric and symmetric matrices, respectively, and $\delta,\mathcal D(a),\beta$ are defined in Theorem \ref{T3.1}. Then, there exists $x_i^\infty \in \bbs^{2d-1}$ such that
\[
\lim_{t\to\infty} x_j(t) = x_j^\infty,\quad i\in [N].
\]
In addition, if we decompose $x_j^\infty = (y_j^\infty, z_j^\infty) \in \bbr^{d}\times \bbr^d$, then we have
\[
|\langle x_i^\infty,x_j^\infty\rangle|^2 + |\langle z_i^\infty,y_j^\infty\rangle - \langle y_i^\infty,z_j^\infty\rangle|^2 =1,\quad i,j\in [N].
\]
\end{theorem}

\begin{remark}
If we observe
\begin{align*}
|x_i^0 - x_j^0| <  | (x_i^{1,0},\cdots, x_i^{d,0})-(x_j^{1,0},\cdots, x_j^{d,0})| + |(x_i^{d+1,0},\cdots, x_i^{2d,0})-(x_j^{d+1,0},\cdots, x_j^{2d,0})|,
\end{align*}
then the assumption on initial data in Theorem \ref{T3.2}(iii) merely becomes 
\[
\max_{1\leq i,j\leq N} |x_i^0 - x_j^0|\ll1.
\]
\end{remark} 

\subsection{Numeric simulations} In this subsection, we conduct numerical simulations to support our theoretical results presented in  Theorem \ref{T3.1}. For the numerical method,  the Runge--Kutta fourth-order method is employed to discretize system \eqref{C-10} with the time step $\Delta t=0.01$:
\begin{equation}\label{NU}
\dot{w}_j=\mathrm{i}a_jw_j+\frac{\kappa}{N}\sum_{k=1}^N\left(w_k-\mathrm{Re}\left(\langle w_k, w_j\rangle\right) w_j\right),\quad t>0,\quad j\in [N].
\end{equation}
 We fix the following parameters
\begin{align}\label{NU-p}
d=2,\quad N=4,\quad (a_1, a_2, a_3, a_4)=\left(-0.2831,   -0.0196,    0.0708,    0.2318\right),
\end{align}
where $\{a_i \}$ are chosen to satisfy zero-average with increasing order, i.e., $a_1<a_2<a_3<a_4$. In addition, initial data $\{w_i^0\}_{i=1}^4$  are randomly chosen from $\bbh\bbs^1 \subseteq \bbc^2$. 
\begin{align*}
&w_1^0=(0.3895 - 0.9178i,  -0.0770 + 0.0004i),\quad
w_2^0=(-0.5190 + 0.4832i,  -0.5974 - 0.3746i),\\
&w_3^0=(-0.2123 - 0.8137i,  -0.5232 + 0.1381i),\quad
w_4^0=(0.1684 + 0.3132i,   0.6192 - 0.7001i).
\end{align*}


\begin{figure}[h]
\centering
\mbox{
\subfigure[Evolution of $\dot w_i$]{  
\includegraphics[width=0.5\textwidth]{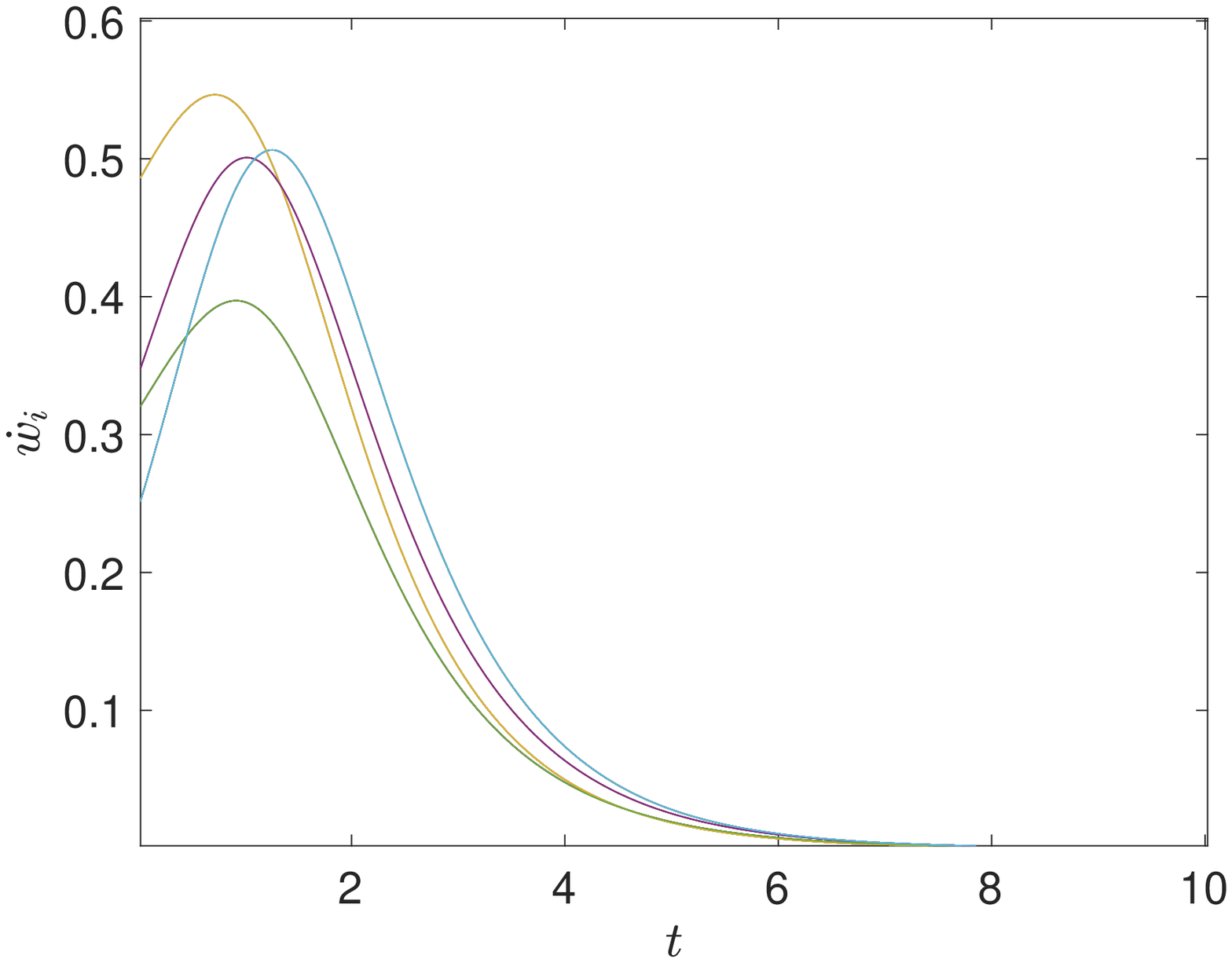}\label{fig:1a}}
\subfigure[ Evolution of $\log|\dot w_i|$]{  
\includegraphics[width=0.5\textwidth]{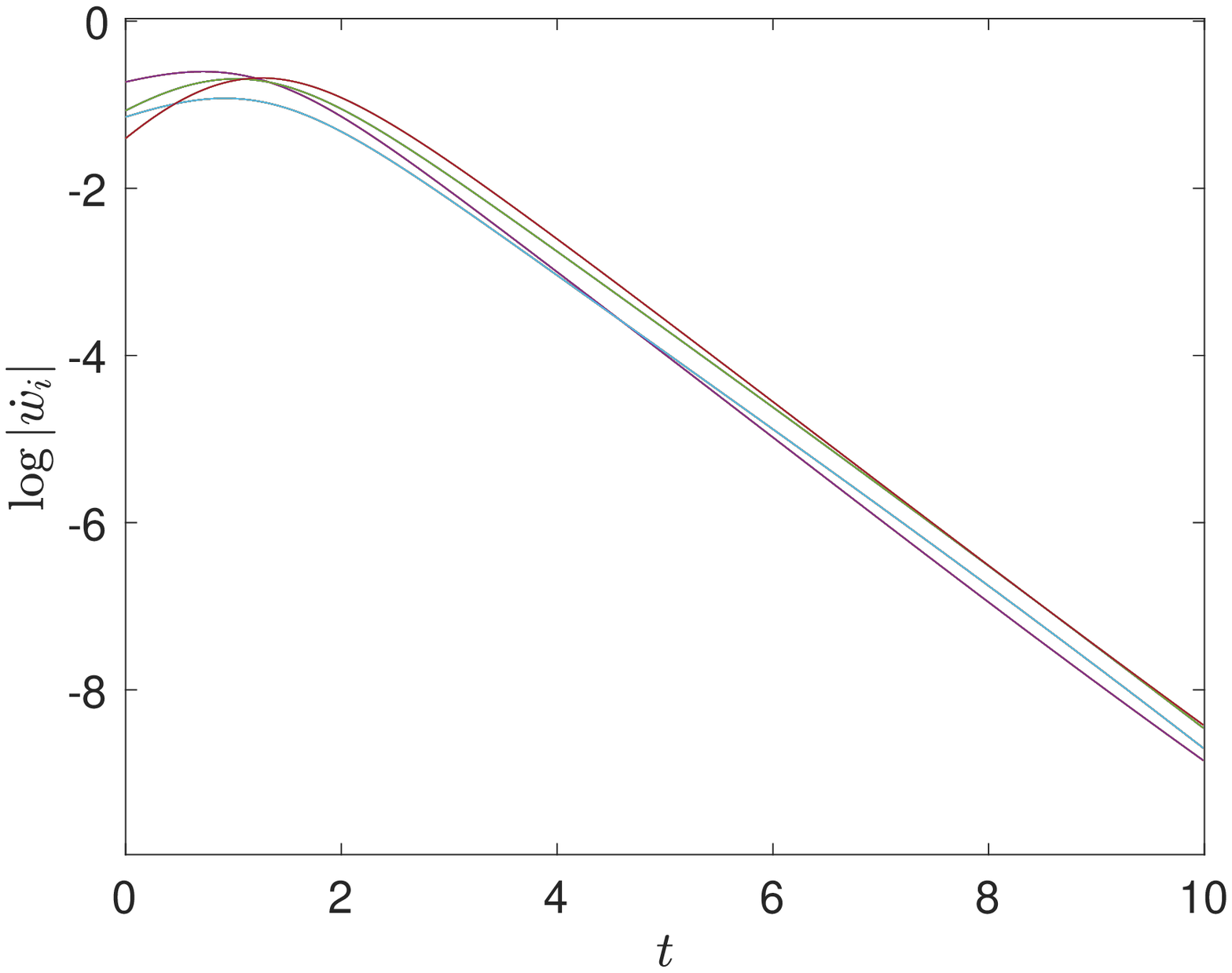}\label{fig:1b}} }
\mbox{
\subfigure[Evolution of $|\langle w_i,w_j\rangle|$]{  
\includegraphics[width=0.5\textwidth]{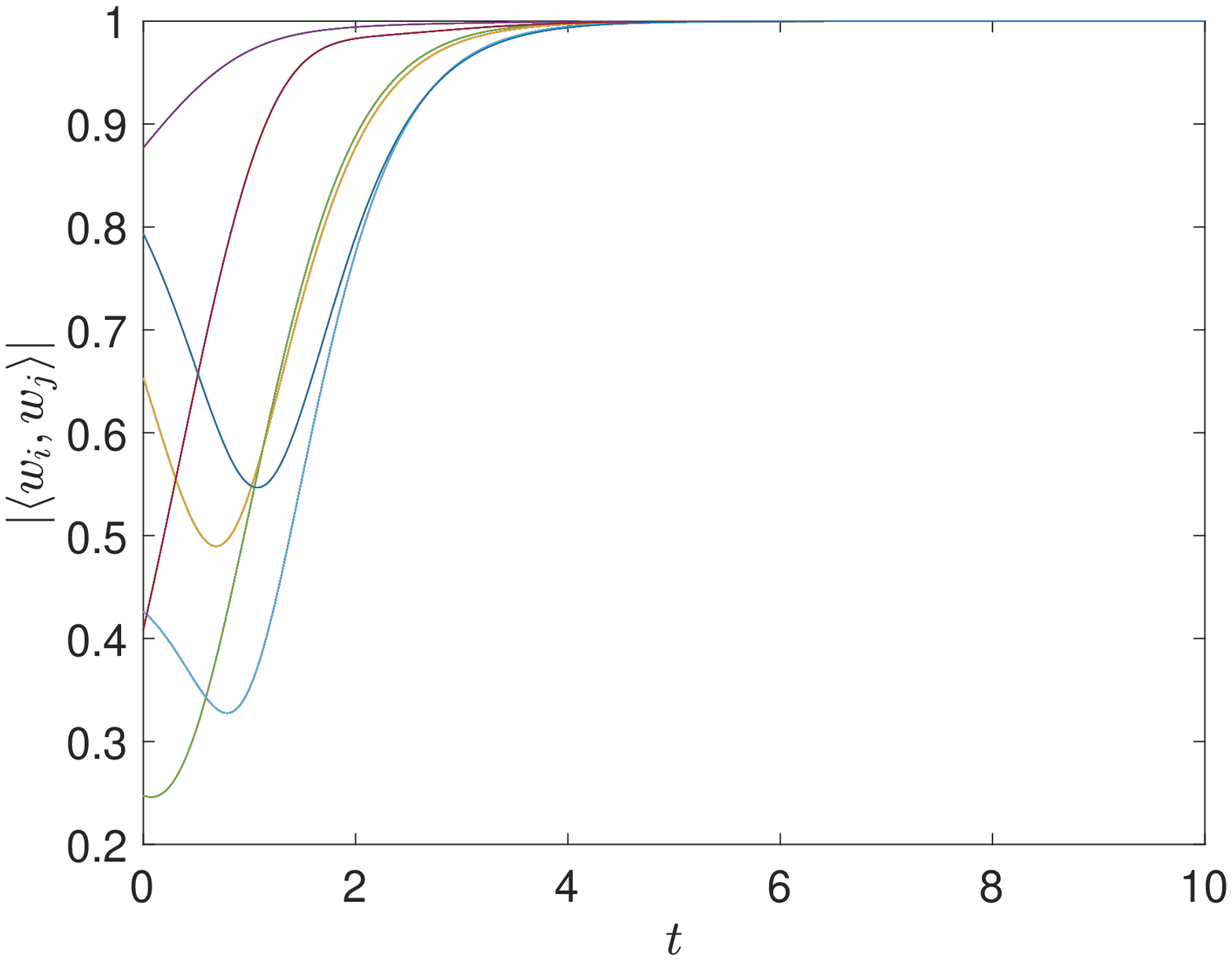}\label{fig:1c}}
\subfigure[ Evolution of geodesic distances]{  
\includegraphics[width=0.5\textwidth]{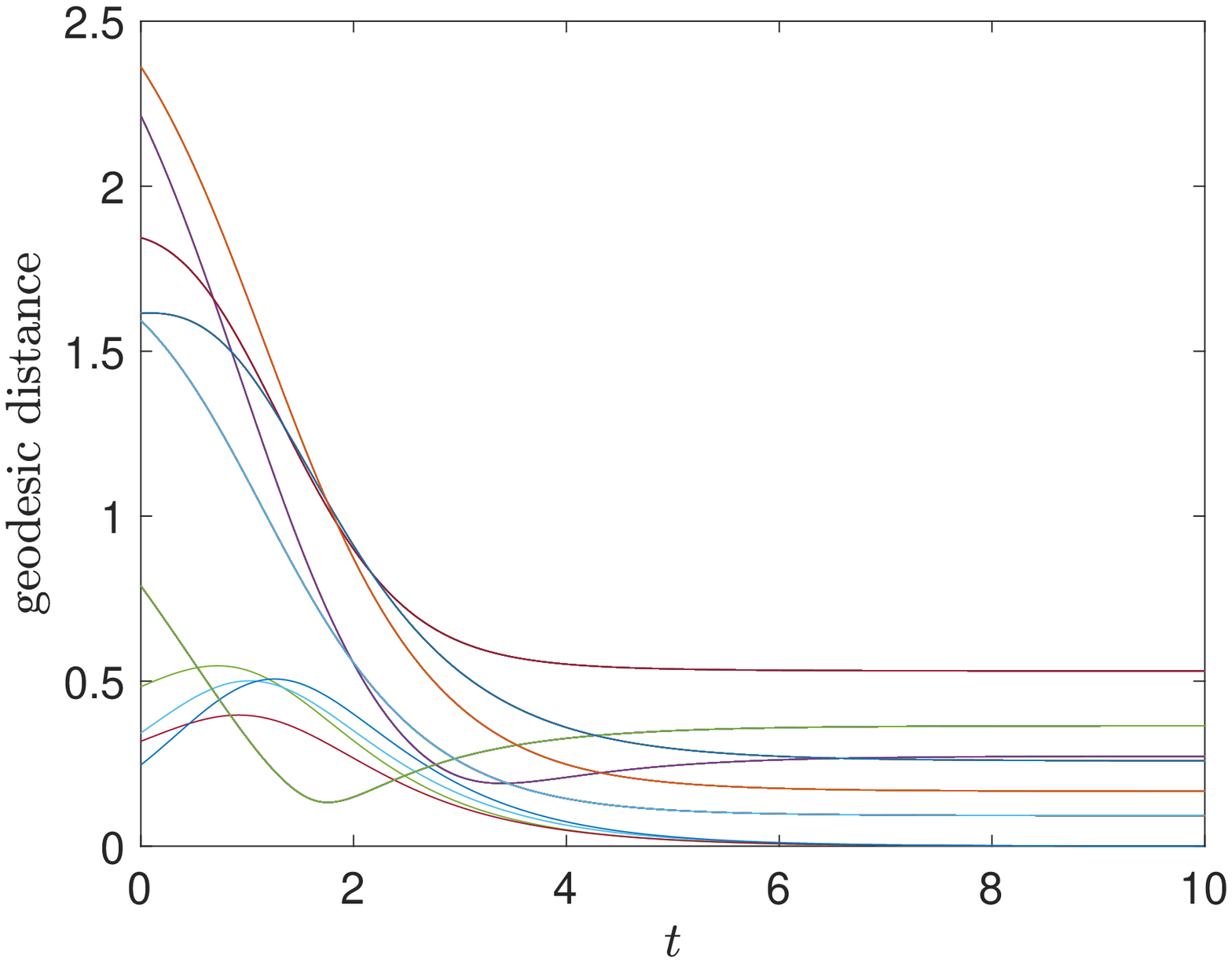}\label{fig:1d}} }
\caption{Numeric simulations of system \eqref{NU} with parameters \eqref{NU-p}.}  \label{fig1}
\end{figure}

In Figure \ref{fig:1a} and Figure \ref{fig:1b}, we see that $\dot w_i$ converges to zero exponentially and this shows that each $w_i$ converges to definite value. From the numerical simulation, we find the final configurations for $w_i$: 
\begin{align*}
&w_1^\infty=(-0.5002 - 0.4755i,  -0.3890 - 0.6102i),\quad
w_2^\infty=(-0.3537 - 0.5926i,  -0.2104 - 0.6924i),\\
&w_3^\infty=(-0.2977 - 0.6226i,  -0.1458 - 0.7089i),\quad
w_4^\infty=(-0.1906 - 0.6633i,  -0.0265 - 0.7232i),
\end{align*}
and direct calculation yields 
\[
|\langle w_i^\infty, w_j^\infty\rangle|=1\quad\forall~i, j\in\{1, 2, 3, 4\}.
\]
In Figure \ref{fig:1c}, we also provide temporal evolutions for $|\langle w_i, w_j\rangle|$, and this coincides with the calculations above. On the other hand, we define  a $4\times 4$ matrix $\Delta\Theta^\infty=(\Delta\Theta^\infty_{ij})_{i,j=1}^4$ where  $\Delta\Theta^\infty_{ij}$ is a geodesic distance between $w_i^\infty$ and $w_j^\infty$ on the hermitian sphere $\mathbb{HS}^1\simeq \mathbb{S}^{3}$, i.e., 
\[
\Delta\Theta^\infty_{ij}=\arccos\left(1-\frac{|w_i^\infty-w_j^\infty|^2}{2}\right), \quad\forall~i, j\in\{1, 2, 3, 4\}.
\]
Since $\Delta\Theta^\infty$ is calculated as 
\[
\Delta\Theta^\infty=\begin{bmatrix}
         0    &0.2726   & 0.3647    &0.5310\\
    0.2726         &0    &0.0922    &0.2584\\
    0.3647    &0.0922         &0    &0.1662\\
    0.5310    &0.2584   &0.1662     &    0
\end{bmatrix},
\]
one can easily check
\begin{align*}
\begin{cases}
\Delta\Theta^\infty_{13}&=\Delta\Theta^\infty_{12}+\Delta\Theta^\infty_{23},\\
\Delta\Theta^\infty_{24}&=\Delta\Theta^\infty_{23}+\Delta\Theta^\infty_{34},\\
\Delta\Theta^\infty_{14}&=\Delta\Theta^\infty_{12}+\Delta\Theta^\infty_{23}+\Delta\Theta^\infty_{34}.
\end{cases}
\end{align*}
This shows that $w_1^\infty$, $w_2^\infty$, $w_3^\infty$, and $w_4^\infty$ are on the same geodesic in order of $\{a_i\}$.  Furthermore, we can check that $\Theta=(\theta_i)_{i=1}^N$ defined by
\begin{equation} \label{C-50}
\theta_i=\varphi+\Theta_{1i}\quad\forall~i\in\{1, 2, 3, 4\},
\end{equation} 
is a steady state of system \eqref{C-10-1}. Finally, we conclude that \eqref{C-50}  gives the relationship between \eqref{C-10} and \eqref{C-10-1}.

\section{Convergence toward equilibrium on the complex unit sphere} \label{sec:4} 
\setcounter{equation}{0}
In this section, we investigate the detailed asymptotic behavior of the Schr\"odinger--Lohe model and the Lohe Hermitian model whose states are represented in general by complex numbers. We continue to apply our dimension reduction method. 

\subsection{Schr\"odinger--Lohe model} 
In this subsection, we consider the heterogeneous  Schr\"odinger--Lohe model with a specific structure:
\[
V_j(x) = V(x) + a_j,\quad a_j\in \bbr,\quad  j\in [N].
\]
In other words, external potentials are given as a (small) perturbation $\{a_j\}$ of a common potential $V(x)$. Furthermore, as in the previous section, we assume that the perturbation $a_j$ has zero average.  Then, our model reads as
\[
\mi \p_t \psi_j = \left( -\frac12 \Delta + V\right)\psi_j + a_j \psi_j + \frac{\mi\kp}{ N} \sum_{k=1}^N ( \psi_k - \langle \psi_j,\psi_k\rangle \psi_j).
\]
Since $e^{-\mi ( -\frac12\Delta + V)}$ is a unitary operator, it follows from the solution splitting property that we may assume $\psi_j = \psi_j(t)$ satisfying
\begin{equation} \label{D-1}
\dot \psi_j =  \mi a_j \psi_j + \frac\kp N \sum_{k=1}^N ( \psi_k - \langle \psi_j,\psi_k\rangle \psi_j).
\end{equation}
Although $\psi_j(t) \in \bbc$, it can be extend to $\bbc^d$ without any effort. In this case, the inner product is given as the standard one in $\bbc^d$. As discussed in the previous section, we associate the Kuramoto model whose natural frequencies are determined as $\{a_j\}$ with zero-average natural frequency:
\begin{equation} \label{D-2}
\dot \theta_j = a_j + \frac{2\kp}{N} \sum_{k=1}^N \sin (\theta_k-\theta_j).
\end{equation} 
It is worthwhile mentioning that the coupling strength in \eqref{D-2} is doubled compared to the (original) Kuramoto (see Remark \ref{R4.1}). We define 
\[
\xi_j(t):= \psi_j(t) e^{-\mi \theta_j(t)}.
\]

\begin{lemma}
Let $\psi_j$ and $\theta_j$ be solutions to \eqref{D-1} and \eqref{D-2}, respectively. Then, $\xi _j= \psi_j e^{-\mi \theta_j}$ satisfies
\begin{equation} \label{D-4}
\dot \xi_j = \frac\kp N \sum_{k=1}^N \Big[ e^{\mi(\theta_k - \theta_j)} (\xi_k -\xi_j)- \langle     \xi_j,e^{\mi(\theta_k-\theta_j)}(\xi_k -\xi_j)\rangle \xi_j \Big],\quad t>0,\quad j\in [N].
\end{equation}
In addition, temporal evolution of  $h_{ij}:= \langle \xi_i,\xi_j\rangle$ is given by
\begin{align} \label{D-5}
\begin{aligned}
\dot h_{ij} &=\frac\kp N\sum_{k=1}^N  \Big[ e^{\mi (\theta_k - \theta_i)} (h_{kj} - h_{ij}) + e^{\mi (\theta_i - \theta_k)} (1-h_{ik})h_{ij} \\
&\hspace{2cm} + e^{\mi (\theta_j - \theta_k)} (h_{ik} - h_{ij}) + e^{\mi (\theta_k - \theta_j)} (1-h_{kj})h_{ij} \Big].
\end{aligned}
\end{align} 
\end{lemma}

\begin{proof}
It directly follows from straightforward calculation. 
\end{proof}

We denote the maximal diameter 
\[
\mathcal D(\mathcal H(t)):= \max_{1\leq i,j\leq N} |1-h_{ij}(t)|,\quad h_{ij}(t) = \langle \xi_i,\xi_j\rangle(t),\quad t>0. 
\]
Since the following relation holds due to the unit modulus property of $\xi_i$, 
\[
|\xi_i- \xi_j|^2 = 2(1-\textup{Re}\langle \xi_i,\xi_j\rangle)   \leq 2|1-h_{ij}|,
\]
it suffices to focus on $1-h_{ij}$ instead of the relative distance $|\xi_i-\xi_j|^2$. Below, we derive a differential inequality for $\mathcal D(\mathcal H)$. 

\begin{lemma}
Let $\{\xi_i\}$ be a solution to \eqref{D-4}. Then, the maximal diameter $\mathcal D(\mathcal H)$ satisfies 
\begin{equation} \label{D-6}
\frac\d\dt\mathcal D(\mathcal H) \leq   -2\kp \left( 1-6\sin \frac{\mathcal D(\Theta)}{2} \right)\mathcal D(\mathcal H)+ 2\kp \mathcal D(\mathcal H)^2,\quad t>0. 
\end{equation}

\end{lemma}

\begin{proof}
We use \eqref{D-5} to find 
\begin{align} \label{D-6-1}
\begin{aligned}
 \frac\d\dt(1-h_{ij})& = -2\kp(1-h_{ij}) + \frac\kp N \sum_{k=1}^N (1-h_{ik} + 1-h_{kj})(1-h_{ij}) \\
& \hspace{0.5cm}- \frac\kp N \sum_{k=1}^N \Big[ (e^{\mi (\theta_k - \theta_i)}-1) (h_{kj}- h_{ij}) +  (e^{\mi (\theta_j - \theta_k)}-1) (h_{ik}- h_{ij}) \\
&\hspace{2.2cm} +  (e^{\mi (\theta_i - \theta_k)}-1)(1-h_{ik})h_{ij} +  (e^{\mi (\theta_k - \theta_j)}-1) (1-h_{kj})h_{ij}\Big].
\end{aligned}
\end{align}
Then, \eqref{D-6-1} and \eqref{C-18} give
\begin{align*}
\frac\d\dt\mathcal D(\mathcal H) &\leq -2\kp \mathcal D(\mathcal H) + 2\kp \mathcal D(\mathcal H)^2 + 12\kp \sin \frac{\mathcal D(\Theta)}{2} \mathcal D(\mathcal H) \\
& = -2\kp \left( 1-6\sin \frac{\mathcal D(\Theta)}{2} \right)\mathcal D(\mathcal H)+ 2\kp \mathcal D(\mathcal H)^2
\end{align*}

\end{proof} 

We are now ready to provide a sufficient condition leading to the zero convergence of $\mathcal D(\mathcal H)$ under a large coupling strength regime. 
\begin{proposition} \label{P4.1} 
Suppose that the coupling strength (sufficiently large in the sense below) and natural frequencies satisfy 
\begin{equation*} \label{D-10}
\kp > \frac32\mathcal D(a) ,\quad \sum_{i=1}^N a_i =0,
\end{equation*} 
and that there exists a positive constant $\delta \in (0,\frac13)$ such that
\begin{equation} \label{D-11}
 \mathcal D(\Theta^0)< \delta <\frac{\mathcal D(a)}{2\kp}<\frac13 \quad \textup{and}\quad \mathcal D(\mathcal H^0)< 1-3\delta .
\end{equation} 
Let $\{\xi_j\}$ and $\{\theta_j\}$ be solutions to \eqref{D-4} and \eqref{D-2}, respectively. Then, the maximal diameter converges to zero exponentially:
\[
\lim_{t\to\infty} \mathcal D( \mathcal H(t))=0.
\]
Moreover, there exists $\xi^\infty \in \bbh\bbs^{d-1}$ such that
\[
\lim_{t\to\infty} \xi_j(t) = \xi^\infty,\quad j\in [N]. 
\]

\end{proposition}

\begin{proof}
Since initial data satisfy  $\eqref{D-11}_1$, we use Proposition \ref{P2.1} to see 
\[
 \mathcal D(\Theta(t)) <\delta, \quad t>0. 
\]
Then, \eqref{D-6} becomes 
\[
\frac\d\dt\mathcal D(\mathcal H) \leq -2\kp ( 1- 3\delta)  \mathcal D(\mathcal H) +2\kp \mathcal D(\mathcal H)^2, \quad t>0. 
\]
Again, initial condition $\eqref{D-11}_2$ gives the desired exponential decay of $\mathcal D(\mathcal H)$ toward zero. For the second assertion, we observe from \eqref{D-4} 
\[
|\dot \xi_j | \leq 2\sqrt2\kp \sqrt{\mathcal D(\mathcal H)},
\]
which gives $|\dot \xi_j| \in L^1(\bbr_+)$. Thus, we have
\[
\lim_{t\to\infty} \xi_j(t) = \xi_j^0 + \int_0^ \infty \dot \xi_j(s) \d s  =: \xi_j^\infty.
\]
Since $\mathcal D(\mathcal H)$ tends to zero, we can easily deduce that all $\{\xi_i^\infty\}$ should be same. This shows the existence of the desired common vector $\xi^\infty$. 
\end{proof}

In Proposition \ref{P4.1}, we have shown that all $\xi_i$ collapse to a common stationary point $\xi^\infty$ under some initial framework. We are now concerned with the asymptotic behavior of \eqref{D-1}, which is our main interest in this subsection. For this, we denote
\[
\mathcal D(\Psi(t)) := \max_{1\leq i,j\leq N} |1-\langle \psi_i, \psi_j\rangle(t)|,\quad t>0. 
\]

\begin{theorem}
Suppose that the coupling strength and natural frequencies satisfy 
\[
\kp >\frac32\mathcal D(a), \quad \sum_{i=1}^N a_i=0,
\]
and that there exists $\delta \in (0,\beta)$ such that
\begin{equation} \label{D-15}
0<\delta <\frac{\mathcal D(a)}{\kp}<\beta, \quad \mathcal D(\Psi^0)< 1-3\delta-2\sin\frac\delta2,
\end{equation} 
where $\beta \approx 0.250$ is a unique positive root for $1-3s = 2\sin \frac s2$. 
Let $\{\psi_j\}$ be a solution to \eqref{D-1}. Then, there exists $\psi_j^\infty \in \bbh\bbs^{d-1}$ such that 
\[
\lim_{t\to\infty} \psi_j(t) = \psi_j^\infty,\quad j\in [N]. 
\]
In addition, such $\{\psi_j^\infty\}$ satisfy
\[
|\langle \psi_i^\infty, \psi_j^\infty\rangle|=1,\quad i,j\in [N].
\]
\end{theorem}

\begin{proof}

In order to use Proposition \ref{P4.1}, we observe
\begin{equation} \label{D-16}
|1-\langle \xi_i^0,\xi_j^0\rangle | = |1-\langle \psi_i^0,\psi_j^0\rangle + \langle \psi_i^0,\psi_j^0\rangle (1-e^{\mi(\theta_j^0-\theta_i^0)})| \leq \mathcal D(\Psi^0) + 2\sin\frac{\mathcal D( \Theta^0)}{2} .
\end{equation}
Since initial diameter $\mathcal D(\Psi^0)$ satisfies \eqref{D-15}, initial conditions for $\mathcal D(\mathcal H^0)$ and $\mathcal D(\Theta^0)$ in \eqref{D-11} hold. Thus, we use Proposition \ref{P4.1} and Proposition \ref{P2.1} to conclude that there exist $\xi^\infty \in \bbh\bbs^{d-1}$ and $\theta_j^\infty$ such that
\[
\lim_{t\to\infty} \psi_j(t)  = \xi^\infty e^{\mi\theta_j^\infty}=:\psi_j^\infty
\]
which shows that $\langle \psi_i^\infty,\psi_j^\infty\rangle$ has unit modulus:
\[
\langle\psi_i^\infty,\psi_j^\infty \rangle = \langle \xi^\infty e^{\mi\theta_i^\infty},\xi^\infty e^{\mi\theta_j^\infty}\rangle = e^{\mi(\theta_i^\infty-\theta_j^\infty)}.
\]

\end{proof}

\begin{remark} \label{R4.1} 
We here mention that the Kuramoto model \eqref{D-2} with doubled coupling strength was  formally derived from \eqref{D-1} by using the ansatz:
\[
\psi_j(t) = e^{\mi\theta_j(t)}.
\]
Similar to Remark \ref{R3.2}, all information for the natural frequency part of \eqref{D-1} have been encoded to the Kuramoto model. Thus, we would say that the asymptotic behavior of  high-dimensional quantity $\psi_j$ can be determined by the one of one-dimensional quantity $\theta_j$. 
\end{remark} 

\begin{remark}
In fact, the authors of \cite{H-H-K21} have shown that $\langle \psi_i,\psi_j\rangle$ converges to a definite value. However, there is no further information on the convergent values. In addition, the existence of a limit for $\psi_j$ was not provided. By applying the dimension reduction method, we are here able to show that such convergent values have unit modulus.

\end{remark}

\subsection{Lohe Hermitian sphere model}
In this subsection, we turn to the Lohe Hermitian sphere model on  $\bbh\bbs^{d-1}$ which reads as
\begin{equation} \label{D-50}
\dot z_j = \mi a_j z_j + \frac{\kp_0}{ N }\sum_{k=1}^N (z_k - \langle z_j,z_k\rangle z_j) + \frac{\kp_1}{N} \sum_{k=1}^N ( \langle z_k,z_j\rangle - \langle z_j,z_k\rangle ) z_j,\quad t>0,\quad j \in [N].
\end{equation} 
As we have done, we couple the Kuramoto model whose natural frequencies are $\{a_j\}$ with zero average: 
\begin{equation} \label{D-51}
\dot \theta_j = a_j + \frac{2(\kp_0+\kp_1)}{N} \sum_{k=1}^N \sin (\theta_k-\theta_j),
\end{equation}
and define an auxiliary variable
\[
\xi_j (t) := z_j(t) e^{-\mi\theta_j(t)}.
\]

\begin{lemma}
Let $z_j$ and $\theta_j$ be solutions to \eqref{D-50} and \eqref{D-51}, respectively. Then, $\xi_j = z_j e^{-\mi\theta_j}$ satisfies 
\begin{align} \label{D-53}
\begin{aligned}
\dot \xi_j &=\frac{\kp_0}{ N} \sum_{k=1}^N \Big[ e^{\mi(\theta_k - \theta_j)} (\xi_k -\xi_j)- \langle     \xi_j,e^{\mi(\theta_k-\theta_j)}(\xi_k -\xi_j)\rangle \xi_j \Big]  \\
& \hspace{0.5cm}+\frac{\kp_1}{N} \sum_{k=1}^N \Big[ \langle e^{\mi(\theta_k - \theta_j)}(\xi_k - \xi_j), \xi_j\rangle - \langle \xi_j, e^{\mi(\theta_k - \theta_j)}(\xi_k - \xi_j) \rangle \Big] \xi_j .
\end{aligned}
\end{align} 
In addition, temporal evolution of $h_{ij}:=\langle \xi_i,\xi_j\rangle$ is given as
\begin{align} \label{D-54}
\begin{aligned}
\dot h_{ij} &=\frac{\kp_0}{ N}\sum_{k=1}^N  \Big[ e^{\mi (\theta_k - \theta_i)} (h_{kj} - h_{ij}) + e^{\mi (\theta_i - \theta_k)} (1-h_{ik})h_{ij} \\
&\hspace{2cm} + e^{\mi (\theta_j - \theta_k)} (h_{ik} - h_{ij}) + e^{\mi (\theta_k - \theta_j)} (1-h_{kj})h_{ij} \Big] \\
& \hspace{0.5cm}+ \frac{\kp_1}{N} \sum_{k=1}^N \Big[ e^{\mi(\theta_k- \theta_i)}(h_{ki}-1)h_{ij} - e^{\mi(\theta_i- \theta_k)}(h_{ik}-1)h_{ij} \\
&\hspace{2.5cm} + e^{\mi(\theta_j - \theta_k)}(h_{jk}-1)h_{ij} - e^{\mi(\theta_k - \theta_j)}(h_{kj}-1)h_{ij}\Big].
\end{aligned}
\end{align}
\end{lemma}

\begin{proof}
The proof directly follows from straightforward calculation. 
\end{proof}

We denote the maximal diameter 
\[
\mathcal D(\mathcal H(t)) :=\max_{1\leq i,j\leq N} |1-h_{ij}(t)|,\quad h_{ij} = \langle \xi_i,\xi_j\rangle(t),\quad t>0. 
\]
As for the Schr\"odinger--Lohe model, we consider the dynamics of the maximal quantities among $|1-h_{ij}(t)|$ instead of, for instance, relative distance $|\xi_i-\xi_j|^2$. Below, we find a differential inequality for $\mathcal D(\mathcal H)$. 

\begin{lemma}
Let $\{\xi_j\}$ be a solution to \eqref{D-53}. Then, the maximal diameter $\mathcal D( \mathcal H)$ satisfies
\begin{equation} \label{D-56}
\frac\d\dt \mathcal D(\mathcal H) \leq -2\left(   \kp_0-2\kp_1  - (6\kp_0 +4\kp_1 )\sin \frac{\mathcal D(\Theta)}{2}      \right) \mathcal D(\mathcal H) + 2(\kp_0 + 2\kp_1) \mathcal D(\mathcal H)^2 .
\end{equation} 
\end{lemma}

\begin{proof}
We use \eqref{D-54} to find 
\begin{align*}
\frac\d\dt (1-h_{ij}) & = -2\kp_0 (1-h_{ij}) + \frac{\kp_0}{N} \sum_{k=1}^N ( 1-h_{ik} + 1-h_{kj})(1-h_{ij}) \\
&\hspace{0.5cm}+\frac{\kp_1}{N} \sum_{k=1}^N ( 1-h_{ik} + 1-h_{kj} -(1-h_{ki})-(1-h_{jk}))(1-h_{ij}) \\
&\hspace{0.5cm} + \frac{\kp_1}{N} \sum_{k=1}^N ( 1-h_{ik} + 1-h_{kj} -(1-h_{ki})-(1-h_{jk})) \\
&\hspace{0.5cm}- \frac{\kp_0}{ N}\sum_{k=1}^N  \Big[ (e^{\mi (\theta_k - \theta_i)} -1)(h_{kj} - h_{ij}) + (e^{\mi (\theta_i - \theta_k)} -1)(1-h_{ik})h_{ij} \\
&\hspace{2.5cm} +( e^{\mi (\theta_j - \theta_k)}-1)(h_{ik} - h_{ij}) + (e^{\mi (\theta_k - \theta_j)} -1)(1-h_{kj})h_{ij} \Big] \\
& \hspace{0.5cm}- \frac{\kp_1}{N} \sum_{k=1}^N \Big[( e^{\mi(\theta_k- \theta_i)}-1)(h_{ki}-1)h_{ij} - (e^{\mi(\theta_i- \theta_k)}-1)(h_{ik}-1)h_{ij} \\
&\hspace{2.5cm} + (e^{\mi(\theta_j - \theta_k)}-1) (h_{jk}-1)h_{ij} -( e^{\mi(\theta_k - \theta_j)}-1)(h_{kj}-1)h_{ij}\Big].
\end{align*}
Then, we derive the desired inequality: 
\begin{align*}
\frac\d\dt \mathcal D(\mathcal H) &\leq  -2\left(   \kp_0-2\kp_1  - (6\kp_0 +4\kp_1 )\sin \frac{\mathcal D(\Theta)}{2}      \right) \mathcal D(\mathcal H) + 2(\kp_0 + 2\kp_1) \mathcal D(\mathcal H)^2 .
\end{align*}
\end{proof}

We are now ready to provide a sufficient framework under which the maximal diameter  $\mathcal D(\mathcal H)$ decays to zero. 

\begin{proposition} \label{P4.2}
Suppose that the coupling strength (sufficiently large in the sense below) and natural frequencies satisfy 
\begin{equation*} \label{D-58}
\kp_0>2\kp_1,\quad \frac{2(\kp_0-2\kp_1)(\kp_0+\kp_1)}{3\kp_0+2\kp_1}>\mathcal D(a),\quad \sum_{j=1}^N a_j=0,
\end{equation*}
and that there exists a positive constant $\delta \in (0, \frac{2\kp_0-4\kp_1}{6\kp_0+4\kp_1})$ such that
\begin{equation} \label{D-59}
 \mathcal D(\Theta^0)< \delta <\frac{\mathcal D(\Omega)}{2(\kp_0+\kp_1)}<\frac{2\kp_0-4\kp_1}{6\kp_0+4\kp_1}           \quad \textup{and}\quad \mathcal D(\mathcal H^0)< \frac{\kp_0-2\kp_1 - (3\kp_0+2\kp_1)\delta}{\kp_0 + 2\kp_1}.
\end{equation}
Let $\{\xi_j\}$ and $\{\theta_j\}$ be solutions to \eqref{D-53} and \eqref{D-51}, respectively. Then, the maximal diameter converges to zero with an exponential rate:
\[
\lim_{t\to\infty} \mathcal D(\mathcal H(t)) =0.
\]
Moreover, there exists $\xi^\infty \in \bbh\bbs^{d-1}$ such that
\[
\lim_{t\to\infty} \xi_j(t) = \xi^\infty,\quad j\in[N]. 
\]
\end{proposition}

\begin{proof}
Since initial data satisfy $\eqref{D-59}_1$, inequality \eqref{D-56} becomes
\[
\frac\d\dt \mathcal D(\mathcal H) \leq -2(\kp_0-2\kp_1 - (3\kp_0+2\kp_1) \delta)  \mathcal D(\mathcal H) + 2(\kp_0+2\kp_1) \mathcal D(\mathcal H)^2  ,\quad t>0.
\]
Then, initial condition $\eqref{D-59}_2$ yields the desired exponential decay of $\mathcal D( \mathcal H)$. For the second assertion, we notice from \eqref{D-53}  that 
\[
|\dot\xi_j| \leq 2\sqrt2(\kp_0+\kp_1)  \sqrt{\mathcal D( \mathcal H)},
\]
which yields $|\dot \xi_j| \in L^1(\bbr_+)$. Thus, 
\[
\lim_{t\to\infty} \xi_j(t) = \xi_j^0 + \int_0^\infty \dot \xi_j (s) \d s =:\xi_j^\infty,
\]
and it directly follows from the zero convergence of $\mathcal D(\mathcal H)$ that all $\{\xi_j^\infty\}$ should be same. Hence, existence of $\xi^\infty$ is shown. 
\end{proof}

By virtue of Proposition \ref{P4.2}, we see that all $\{\xi_i\}$ collapse to a common stationary point $\xi^\infty$ under a well-prepared initial framework. We are now concerned with the asymptotic behavior of \eqref{D-50}, which is our main interest in this subsection. For this, we denote
\[
\mathcal D(\mathcal Z(t)):= \max_{1\leq i,j\leq N} |1-\langle z_i,z_j\rangle(t)|,\quad t>0.
\]

\begin{theorem}
Suppose that the coupling strength natural frequencies satisfy 
\[
\kp_0>2\kp_1,\quad \frac{2(\kp_0-2\kp_1)(\kp_0+\kp_1)}{3\kp_0+2\kp_1}>\mathcal D(a),\quad \sum_{i=1}^N a_i=0,
\]
and that there exists $\delta \in (0,\beta)$ such that
\begin{equation} \label{D-63}
0<  \delta <\frac{\mathcal D(\Omega)}{2(\kp_0+\kp_1)}<\beta ,\quad \mathcal D(\mathcal Z^0)< \frac{\kp_0-2\kp_1 - (3\kp_0+2\kp_1)\delta}{\kp_0 + 2\kp_1}-2\sin\frac\delta2,
\end{equation} 
where $\beta=\beta(\rho)$ is a unique positive root for the algebraic equation:
\[
\frac{\rho-2 -(3\rho+2)s}{\rho+2} = 2\sin\frac s2,\quad \rho: = \frac{\kp_0}{\kp_1}>2
\]
Let $\{z_j\}$ be a solution to \eqref{D-50}. Then, there exists $z_j^\infty \in \bbh\bbs^{d-1}$ such that
\[
\lim_{t\to\infty} z_j(t) = z_j^\infty,\quad j\in[N].
\]
Moreover, such $\{z_j^\infty\}$ satisfy 
\[
|\langle z_i^\infty,z_j^\infty\rangle|=1,\quad i,j\in [N].
\]
\end{theorem}

\begin{proof}
In order to apply Proposition \ref{P4.2}, we recall \eqref{D-16}:
\begin{equation*}
|1-\langle \xi_i^0,\xi_j^0\rangle | = |1-\langle \psi_i^0,\psi_j^0\rangle + \langle \psi_i^0,\psi_j^0\rangle (1-e^{\mi(\theta_j^0-\theta_i^0)})| \leq \mathcal D(\mathcal Z^0) + 2\sin\frac{\mathcal D( \Theta^0)}{2} .
\end{equation*}
Since initial data $\mathcal D(\mathcal Z^0)$ satisfy \eqref{D-63}, initial conditions for $\mathcal D(\mathcal H^0)$ and $\mathcal D(\Theta^0)$ in \eqref{D-59} hold. Hence, it follows from Proposition \ref{P4.2} and Proposition \ref{P2.1} that  there exist $\xi^\infty \in \bbh\bbs^{d-1}$ and $e^{\mi\theta_j^\infty}$ such that
\[
\lim_{t\to\infty} z_j(t) = \xi^\infty e^{\mi\theta_j^\infty} =: z_j^\infty,
\]
which directly yields that $\langle z_i^\infty,z_j^\infty \rangle$ has unit modulus:
\[
\langle z_i^\infty,z_j^\infty \rangle = \langle \xi^\infty e^{\mi\theta_i^\infty}, \xi^\infty e^{\mi\theta_j^\infty}\rangle = e^{\mi(\theta_i^\infty- \theta_j^\infty)}. 
\]
\end{proof}

\begin{remark}
As in Remark \ref{R4.1}, the Kuramoto model \eqref{D-51} with the coupling strength $2(\kp_0+\kp_1)$ can be formally derived from \eqref{D-50} by using the ansatz:
\[
z_j (t) = e^{\mi\theta_j(t)}.
\]
\end{remark}

\section{Convergence toward equilibrium on the unitary group} \label{sec:5} 
\setcounter{equation}{0}
In this section, we are interested in  the detailed asymptotic behavior of first-order matrix aggregation model, namely the Lohe matrix model:
\[
\dot U_j  = -\mi H_j U_j + \frac{\kp}{2N} \sum_{k=1}^N ( U_k - U_jU_k^\dg U_j), \quad t>0,\quad j\in [N].
\]
In particular, we restrict ourselves to specific natural frequency matrices satisfying
\begin{equation}  \label{E-0}
H_j = -a_j I_d,\quad j\in[N].
\end{equation}
Then, our model reads as
\begin{equation} \label{E-1}
\dot U_j = \mi a_j  U_j + \frac{\kp}{2N} \sum_{k=1}^N ( U_k - U_jU_k^\dg U_j),
\end{equation}
linked with the Kuramoto model by natural frequencies $\{a_j\}$ with zero average:
\begin{equation} \label{E-2}
\dot \theta_j = a_j + \frac\kp N \sum_{k=1}^N \sin (\theta_k- \theta_j) .
\end{equation}
In order to proceed our dimension reduction method, we introduce
\[
V_j (t):= e^{-\mi \theta_j(t)} U_j(t).
\]

\begin{lemma}
Let $U_j$ and $\theta_j$ be solutions to \eqref{E-1} and \eqref{E-2}, respectively. Then, $V_j = e^{-\mi \theta_j}U_j$ satisfies 
\begin{equation} \label{E-3}
\dot V_j= \frac{\kp}{2N} \sum_{k=1}^N \Big[ e^{\mi(\theta_k- \theta_j)} (V_k- V_j) - e^{\mi(\theta_j - \theta_k)} V_j (V_k^\dg - V_j^\dg) V_j\Big].
\end{equation} 
In addition, temporal evolution for $V_iV_j^\dg$ is determined by
\begin{align} \label{E-4}
\begin{aligned}
(V_iV_j^\dg)' & = \frac{\kp}{2N} \sum_{k=1}^N \Big[ e^{\mi(\theta_k - \theta_i)} (V_k V_j^\dg - V_iV_j^\dg) - e^{\mi(\theta_i - \theta_k)} V_i(V_k^\dg - V_i^\dg) V_iV_j^\dg  \\
&\hspace{2cm} +  e^{\mi(\theta_j - \theta_k)} (V_i V_k^\dg - V_iV_j^\dg) + e^{\mi(\theta_k - \theta_j)} V_iV_j^\dg (V_k - V_i) V_j^\dg\Big] .
\end{aligned}
\end{align}
\end{lemma}

\begin{proof}
Direct calculation gives the desired result. 
\end{proof}

We denote the maximal diameter 
\[
\mathcal D(\mathcal V(t)):=\max_{1\leq i,j\leq N} \| V_i(t)- V_j(t)\|_\tF= \max_{1\leq i,j\leq N} \|I_d - V_i(t)V^\dg_j(t)\|_\tF,\quad t>0.
\]
 Below, we derive a differential inequality for $\mathcal D(\mathcal V)$. 

\begin{lemma}
Let $\{V_j\}$ be a solution to \eqref{E-3}. Then, the maximal diameter $\mathcal D( \mathcal V)$ satisfies
\begin{equation} \label{E-7}
\frac\d\dt \mathcal D(\mathcal V) \leq -\kp \left(  1-2\sin \frac{\mathcal D(\Theta)}{2}        \right) \mathcal D(\mathcal V) + \kp \mathcal D(\mathcal V)^2,\quad t>0.
\end{equation}

\end{lemma}

\begin{proof}
We use \eqref{E-4} to rewrite 
\begin{align*}
\frac\d\dt (I_d - V_iV_j^\dg)  &= \frac{\kp}{2N} \sum_{k=1}^N ( -2(I_d- V_iV_j^\dg) + (I_d-V_iV_j^\dg)(I_d-V_kV_j^\dg) +(I_d-V_iV_k^\dg)(I_d-V_iV_j^\dg) )\\
&- \frac{\kp}{2N} \sum_{k=1}^N \Big[ (e^{\mi(\theta_k - \theta_i)} -1)(V_k V_j^\dg - V_iV_j^\dg) - (e^{\mi(\theta_i - \theta_k)}-1) V_i(V_k^\dg - V_i^\dg) V_iV_j^\dg  \\
&\hspace{2cm} +  (e^{\mi(\theta_j - \theta_k)} -1)(V_i V_k^\dg - V_iV_j^\dg) + (e^{\mi(\theta_k - \theta_j)}-1) V_iV_j^\dg (V_k - V_i) V_j^\dg\Big].
\end{align*}
Hence, the desired inequality is derived. 

\end{proof}

Next, we provide a sufficient condition leading to the zero convergence of $\mathcal D( \mathcal V)$.

\begin{proposition} \label{P5.1} 
Suppose that coupling strength and natural frequencies satisfy 
\[
\kp>\mathcal D(a),\quad \sum_{j=1}^N a_j=0,
\]
and that there exists a positive constant $\delta \in (0,1)$ such that
\begin{equation} \label{E-8}
 \mathcal D(\Theta^0)< \delta <\frac{\mathcal D(\Omega)}{\kp}<1, \quad \mathcal D(\mathcal V^0)< 1-\delta
\end{equation}
Let $\{V_j\}$ and $\{\theta_j\}$ be solutions to \eqref{E-3} and \eqref{E-2}, respectively. Then, the maximal diameter converges to zero exponentially:
\[
\lim_{t\to\infty} \mathcal D( \mathcal V(t))=0.
\]
Moreover, there exists a unitary matrix $V^\infty \in \mathbf{U}(d)$ such that
\begin{equation} \label{E-8-1}
\lim_{t\to\infty} V_j(t) = V^\infty,\quad j\in[N].
\end{equation} 
\end{proposition}

\begin{proof}
Since initial data $\Theta^0$ satisfy $\eqref{E-8}_1$, we use Proposition \ref{P2.1} to find
\[
\mathcal D(\Theta(t))<\delta,\quad t>0.
\]
Then, \eqref{E-7} becomes
\[
\frac\d\dt \mathcal D(\mathcal V) \leq -\kp ( 1-\delta)  \mathcal D(\mathcal V) + \kp  \mathcal D(\mathcal V)^2,
\]
which yields the desired zero convergence of $\mathcal D(\mathcal V)$, since initial data $\mathcal V^0$ satisfy condition $\eqref{E-8}_2$.  

On the other hand for the second assertion, we use \eqref{E-3} to attain
\[
\|\dot V_j\|_\tF \leq \kp \sqrt{2d} \mathcal D(\mathcal V),
\]
which gives  $\|\dot V_j\|_\tF \in L^1(\bbr_+)$. Thus, the limit of $V_j(t)$ exists:
\[
\lim_{t\to\infty} V_j(t) = V_j^0 + \int_0^\infty \dot V_j (s) \d s =: V_j^\infty.
\]
Moreover, since $\mathcal D(\mathcal V)$ tends to zero, all $\{V_j^\infty\}$ should be same. Hence, $V^\infty$ in \eqref{E-8-1} exists. 
\end{proof}

In Proposition \ref{P5.1}, we have show that all $V_j$ converge to  constant unitary matrix $V^\infty \in \mathbf{U}(d)$ under some initial conditions. We are now concerned with the asymptotic behavior of \eqref{E-1} as our main interest in this section. For this, we write
\[
\mathcal D(\mathcal U(t)):= \max_{1\leq i,j\leq N} \|I_d - U_i (t) U_j(t)^\dg\|_\tF,\quad t>0. 
\]

\begin{theorem} \label{T5.1} 
Suppose that the coupling strength and natural frequencies satisfy 
\[
\kp > \mathcal D(a),\quad \sum_{j=1}^N a_j=0,
\]
and that there exists $\delta \in (0,\beta)$ such that
\begin{equation} \label{E-10}
0<\delta<\frac{\mathcal D(a)}{\kp}<\beta,\quad \mathcal D(\mathcal U^0)<1-\delta - 2d\sin\frac\delta2,
\end{equation}
where $\beta=\beta(d)$ is a unique positive root for $1-s = 2d\sin\frac s2$. Let $\{U_j\}$ be a solution to \eqref{E-1}. Then, there exists $U_j^\infty \in \mathbf{U}(d)$ such that
\[
\lim_{t\to\infty} U_j(t) = U_j^\infty,\quad j\in [N].
\]
Moreover, the following relation holds:
\begin{equation} \label{E-12}
U_i^\infty U_j^{\infty,\dg} = e^{\mi(\theta_i^\infty- \theta_j^\infty)} I_d,\quad i,j\in[N].
\end{equation}
\end{theorem}

\begin{proof}
In order to apply Proposition \ref{P5.1}, we write
\[
I_d - V_i^0V_j^{0,\dg} = I_d - U_i^0U_j^{0,\dg} +U_i^0U_j^{0,\dg}(1-e^{\mi(\theta_j^0- \theta_i^0)}),
\]
and this gives
\[
\mathcal D(\mathcal V^0) \leq \mathcal D(\mathcal U^0) + 2d\sin\frac{\mathcal D(\Theta^0)}{2} .
\]
Since \eqref{E-10} holds, initial conditions for $\mathcal D(\mathcal V^0)$ and $\mathcal D(\Theta^0)$ in \eqref{E-8} are fulfilled. Hence, Proposition \ref{P5.1} together with Proposition \ref{P2.1} gives $V^\infty \in \mathbf{U}(d)$ and $e^{\mi\theta_j^\infty}$ such that
\[
\lim_{t\to\infty} U_j(t) = e^{\mi\theta_j^\infty} V^\infty  =: U_j^\infty \in \mathbf{U}(d).
\]
Thus, \eqref{E-12} directly follows. 
\end{proof}

\begin{remark}
In \cite{H-R16}, the authors studied the Lohe matrix model:
\[
\dot U_j  = \mi H_j U_j + \frac{\kp}{2N} \sum_{k=1}^N ( U_k - U_jU_k^\dg U_j),
\]
with a general Hermitian matrix $H_j$. As mentioned in Section \ref{sec:2.4}, the authors showed that for a large coupling strength, asymptotic phase-locking emerges. In other words, 
\[
\lim_{t\to\infty} U_i(t) U_j(t)^\dg \quad \textup{exists.}
\]
Although limits of $\{U_iU_j^\dg\}$  exist,  existence of limit for each $U_j$ would not be guaranteed. Thus, if we let
\[
X_i^\infty := \lim_{t\to\infty} (U_iU_1^\dg)(t),\quad i\in [N],
\]
then $\{X_i^\infty\}$ should satisfy 
\begin{equation*}
X_i^\infty \Lambda X_i^{\infty,\dg} = H_i -\frac{\mi\kp}{2N} \sum_{k=1}^N \Big( X_i^\infty X_k^{\infty,\dg} - X_k^\infty X_i^{\infty,\dg} \Big),
\end{equation*}
where $\Lambda$ is an index-independent Hermitian matrix. Our result stated in Theorem \ref{T5.1} says that if $H_j$ has a specific structure \eqref{E-0}, then the  convergence of limit for $U_j$ can be exactly identified. 
\end{remark}


\section{Conclusion} \label{sec:6} 
\setcounter{equation}{0}

%
In this work, we studied the detailed asymptotic dynamics of first-order heterogeneous aggregation models, for instance, real and complex swarm sphere models, the Schr\"odinger-Lohe model, the Lohe Hermitian sphere model, and the Lohe matrix model. It should be noted that the aforementioned models can be regarded as generalized and high-dimensional Kuramoto models and thus can be reduced to the Kuramoto model by a simple ansatz. Since the existence and structure of equilibria for the Kuramoto model are well-known, the key idea, called dimension reduction method, is to decompose the given high-dimensional system into two subsystems: a heterogeneous Kuramoto model and a homogeneous modified model. In this manner, we can establish the existence of a solution to the high-dimensional model. Several numerical examples are provided in order to support theoretical results. There are still remaining issues. In fact, we have assumed that the natural frequency has the degree of freedom $1$. One of the future work is dedicated to generalizing this issue.


\begin{thebibliography}{plain}
 
%
%
%
%
%
%
%
%
%
%
%
%
%
%
%
%
%
%
%
%


 
\bibitem{C-H14} Choi, S.-H. and Ha, S.-Y.: \textit{Complete entrainment of Lohe oscillators under attractive and repulsive couplings.} SIAM Journal on Applied Dynamical Systems. {\bf 13} (2014), 1417-1441.
%
%
%
%
%
%
%
%
%
%
%
%
%
%
%
%
%
%
%
%
%
%
%
%
%
%
%
%
%
%
%
%
%
%
%
%
%
%
%
%
%
%
%
%
%
%
%

\bibitem{H-J-K-P-Z} Ha, S.-Y., Jung, J., Kim, J., Park, J. and Zhang, X.: \textit{Emergent behaviors of the swarmalator model for position-phase aggregation.}  Math. Models Methods Appl. Sci. {\bf29} (2019), 2225-2269. 



\bibitem{H-H-K10} Ha, S.-Y., Ha, T. and Kim, J. H.: \textit{On the complete synchronization for the globally coupled Kuramoto model.} Phys. D {\bf239} (2010), 1692-1700.

\bibitem{H-H-K21} Ha, S.-Y., Hwang, G. and Kim, D.: \textit{Two-point correlation function and its applications to the Schr\"odinger--Lohe type models.} Under review. 

\bibitem{H-P20} Ha, S.-Y. and Park, H.: \textit{From the Lohe tensor model to the Lohe Hermitian sphere model and emergent dynamics.} SIAM J. Appl. Dyn. Syst. {\bf19} (2020), 1312-1342. 



\bibitem{H-R20} Ha, S.-Y. and Ryoo S.-W.: \textit{Asymptotic phase-locking dynamics and critical coupling strength for the Kuramoto model.} Commun. Math. Phys. {\bf377} (2020), 811-857. 

\bibitem{H-R16} Ha, S.-Y. and Ryoo S.-W.: \textit{On the emergence and orbital stability of phase-locked states for the Lohe model.} J. Stat. Phys. {\bf163} (2016), 411-439.

\bibitem{K-K21} Kim, D. and Kim, J.: \textit{On the emergent behavior of the swarming models on the complex sphere.} Stud. Appl. Math. (2021).

\bibitem{Ku1}  Kuramoto, Y.: \textit{Chemical Oscillations, Waves and Turbulence.} Springer-Verlag, Berlin, (1984).

\bibitem{Ku2} Kuramoto, Y.: \textit{International Symposium on Mathematical Problems in Mathematical Physics.} Lecture Notes Theor. Phys. {\bf30} (1975), 420.


\bibitem{Lo10} Lohe, M. A.: \textit{Quantum synchronization over quantum networks.} J. Phys. A {\bf 43} (2010), 465301.

\bibitem{Lo09} Lohe, M. A.: \textit{Non-Abelian Kuramoto model and synchronization.} J. Phys. A {\bf 42} (2009), 395101.


%
\bibitem{O1} Olfati-Saber, R.: \textit{Swarms on sphere: A programmable swarm with synchronous behaviors like oscillator networks.} Proc. of the 45th IEEE conference on Decision and Control (2006), 5060 - 5066.


\bibitem{M-P-G} Markdahl, J., Proverbio, D. and Goncalves, J.: \textit{Robust synchronization of heterogeneous robot swarms on the sphere.} Proc. of the 59th IEEE conference on Decision and Control (2020), 5798-5803.

\bibitem{Wi1} Winfree, A.: \textit{Biological rhythms and the behavior of populations of coupled oscillators.} J. Theor. Biol. {\bf 16} (1967) 15–42.

\bibitem{Wi2} Winfree, A.: \textit{The Geometry of Biological Time}, Springer, New York, 1980.


\bibitem{Z-Z} Zhang, J. and Zhu, J.: \textit{Synchronization of high-dimensional Kuramoto models with nonidentical oscillators and interconnection digraphs.}  IET Control Theory  A (2021).


%
%
%
%
%
%
%
%
%
%
%
%
%
%
%
%
%
%
%
%
%
%
%
%
%
%
%
%
%
%
%
%
%
%
%
%
%

\end{thebibliography}
\end{document}